\documentclass[11pt,a4paper]{article}

\usepackage[margin=27mm]{geometry}
\usepackage{iftex}
\usepackage[T1]{fontenc}
\ifPDFTeX
  \usepackage[utf8]{inputenc}
\fi
\usepackage{lmodern}
\usepackage{microtype}
\usepackage{amsmath,amssymb,amsthm,mathtools}
\usepackage{mathrsfs}
\usepackage{enumitem}
\usepackage{array,booktabs,tabularx}
\usepackage{float}
\usepackage{xcolor}
\usepackage{tikz}
\usetikzlibrary{arrows.meta,calc,positioning}
\usepackage{aliascnt}
\usepackage{hyperref}
\hypersetup{
  colorlinks=true,
  linkcolor=blue!45!black,
  citecolor=blue!45!black,
  urlcolor=blue!55!black,
  pdftitle={
  Entire Logarithmic Signatures of Bounded-Variation Paths in Finite Dimensions
  },
  pdfauthor={Elena Boguslavskaya},
  pdfsubject={A conjugacy classification of bounded-variation signatures with entire logarithm and a matrix spectral-growth theorem},
  pdfkeywords={path signatures, logarithmic signatures, bounded variation, Lyons-Sidorova conjecture, free Lie algebras, spectral invariants, metric trees}
}

\allowdisplaybreaks
\setlist[enumerate]{itemsep=0.3em,topsep=0.5em}
\setlist[itemize]{itemsep=0.3em,topsep=0.5em}

\newtheorem{theorem}{Theorem}[section]

\newaliascnt{proposition}{theorem}
\newtheorem{proposition}[proposition]{Proposition}
\aliascntresetthe{proposition}

\newaliascnt{lemma}{theorem}
\newtheorem{lemma}[lemma]{Lemma}
\aliascntresetthe{lemma}

\newaliascnt{corollary}{theorem}
\newtheorem{corollary}[corollary]{Corollary}
\aliascntresetthe{corollary}

\newaliascnt{conjecture}{theorem}

\aliascntresetthe{conjecture}

\newaliascnt{definition}{theorem}
\newtheorem{definition}[definition]{Definition}
\aliascntresetthe{definition}

\theoremstyle{remark}
\newaliascnt{remark}{theorem}
\newtheorem{remark}[remark]{Remark}
\aliascntresetthe{remark}

\usepackage[nameinlink,capitalise,noabbrev]{cleveref}

\crefname{theorem}{Theorem}{Theorems}
\Crefname{theorem}{Theorem}{Theorems}
\crefname{proposition}{Proposition}{Propositions}
\Crefname{proposition}{Proposition}{Propositions}
\crefname{lemma}{Lemma}{Lemmas}
\Crefname{lemma}{Lemma}{Lemmas}
\crefname{corollary}{Corollary}{Corollaries}
\Crefname{corollary}{Corollary}{Corollaries}
\crefname{conjecture}{Conjecture}{Conjectures}
\Crefname{conjecture}{Conjecture}{Conjectures}
\crefname{definition}{Definition}{Definitions}
\Crefname{definition}{Definition}{Definitions}
\crefname{remark}{Remark}{Remarks}
\Crefname{remark}{Remark}{Remarks}

\newcommand{\R}{\mathbb{R}}
\newcommand{\C}{\mathbb{C}}
\newcommand{\Q}{\mathbb{Q}}
\newcommand{\Z}{\mathbb{Z}}
\newcommand{\Tens}{\mathcal{T}}
\newcommand{\Var}{\operatorname{Var}}
\newcommand{\spec}{\operatorname{spec}}
\newcommand{\tr}{\operatorname{tr}}
\newcommand{\diag}{\operatorname{diag}}
\newcommand{\id}{\operatorname{id}}
\newcommand{\Fix}{\operatorname{Fix}}
\newcommand{\Ad}{\operatorname{Ad}}
\newcommand{\e}{\mathrm{e}}
\newcommand{\dd}{\,\mathrm{d}}
\newcommand{\uLie}{\mathfrak{u}}
\newcommand{\Sl}{\mathscr{S}}
\newcommand{\norm}[1]{\left\lVert #1\right\rVert}
\newcommand{\abs}[1]{\left\lvert #1\right\rvert}
\newcommand{\set}[1]{\left\{#1\right\}}

\title{\textbf{Entire Logarithmic Signatures\\
of Bounded-Variation Paths in Finite Dimensions}}

\author{%
Elena Boguslavskaya\\[0.35em]
\small Mathematics Department\\
\small Brunel University of London\\
\small London, UK\\
\small Corresponding author: \texttt{elena@boguslavsky.net}\\
\small ORCID: \texttt{0000-0001-7347-7115}%
}
\date{}

\begin{document}
\maketitle

\begin{abstract}
We classify signatures of bounded-variation paths in finite-dimensional real normed spaces whose logarithms are entire, in the sense of superexponential homogeneous decay. The zero-first-level fibre is trivial; if the first level is $v\neq0$, the signature is $a\e^v a^{-1}$ with $a$ a bounded-variation signature, and every such conjugate is entire. For a given path, $a$ may be chosen from a prefix. For a tree-reduced representative, a gate-selected prefix gives weak path conjugacy to a line. The proof uses a finite-dimensional spectral-growth statement: if $q\geq1$, $A:\C\to M_q(\C)$ is entire, and $\norm{\e^{A(z)}}\leq C\e^{\tau\abs{z}}$, then the $k$th characteristic-polynomial coefficient of $A(z)$ has degree at most $k$ for $1\leq k\leq q$. Universal matrix isospectrality, resonant developments, metric-tree fixed points and compactness, and Stieltjes--Fourier reconstruction establish the bounded-variation modified Lyons--Sidorova conjecture under a hypothesis imposed only on the whole path.
\end{abstract}



\tableofcontents

\section{Introduction}

\subsection{Entire logarithms as a structural problem}

Let $\gamma:[0,T]\to V$ be a continuous bounded-variation path in a finite-dimensional real normed space, and write $V_{\C}=V\otimes_{\R}\C$ and $\Tens((V))=\prod_{n\geq0}V^{\otimes n}$. Its degree-$n$ signature level and full signature are
\[
S_n(\gamma)
 =\int_{0<t_1<\cdots<t_n<T}
 \dd\gamma_{t_1}\otimes\cdots\otimes\dd\gamma_{t_n},
\qquad
S(\gamma)=\sum_{n\geq0}S_n(\gamma),
\]
with $S_0(\gamma)=1$. A complex-linear map $\Phi:V_{\C}\to M_q(\C)$ extends multiplicatively to each tensor level, denoted $\Phi^{(n)}$; applying these extensions to $S(\gamma)$ gives the endpoint series of the corresponding controlled matrix equation. Such finite-dimensional developments are standard tools in rough path theory and in signature asymptotics \cite{Lyons1998,BGS2020}. Chen's identity makes $S(\gamma)$ multiplicative under concatenation. Its formal logarithm
\[
l:=\log S(\gamma)=\sum_{n\geq1}l_n
\]
is primitive, so each $l_n$ is a homogeneous free Lie polynomial \cite{Chen1957,Chen1958,Reutenauer1993}. With $\norm{\cdot}_{\pi,n}$ denoting the projective tensor norm on $V_{\C}^{\otimes n}$, set
\begin{equation}\label{eq:intro-radius}
R(l)=\left(\limsup_{n\to\infty}\norm{l_n}_{\pi,n}^{1/n}\right)^{-1}.
\end{equation}
We call $l$ \emph{entire} when $R(l)=\infty$, equivalently when its homogeneous components have superexponential decay. This condition is independent of the admissible finite-dimensional crossnorm used; see \cref{lem:norm-independence}. Signature levels have factorial bounds, whereas logarithmic levels need not, so the condition is not automatic for bounded-variation paths.
For $z\in\C$, homogeneous dilation is defined by $\delta_z(\sum a_n)=\sum z^n a_n$. If $R(l)=\infty$ and $\Phi:V_{\C}\to M_q(\C)$ is a matrix development, then
\[
A_{\Phi}(z)=\Phi(\delta_zl)=\sum_{n\geq1}z^n\Phi^{(n)}(l_n)
\]
is an entire matrix function. The proof first extracts spectral information from this family of evaluations and later uses specially chosen sparse developments to recover tensor identities.

A line segment with increment $v$ has signature $\e^v=\sum_{n\geq0}v^{\otimes n}/n!$ and logarithmic signature $v$. Lyons and Sidorova asked whether infinite logarithmic-signature radius forces a reduced bounded-variation path to be straight. They proved several rigidity results and, for every nonconstant path, the lower bound
\[
R(\log S(\gamma))\geq \frac{\kappa}{\Var(\gamma)},
\qquad \log\kappa=-\kappa,
\]
with $\kappa\in(0,1)$ \cite{LyonsSidorova2006}. Chevyrev and Lyons treated positivity for geometric rough paths \cite{ChevyrevLyons2016}, while Friz, Lyons, and Seigal proved that a rectifiable reduced path with polynomial logarithmic signature is a line \cite{FrizLyonsSeigal2024}.

The signature, rather than a chosen parametrisation, is the natural object here. Trivial signature is equivalent to tree-likeness, and bounded-variation paths with the same signature are tree-like equivalent; each class has a unique reduced representative up to translation and increasing reparametrisation \cite{HamblyLyons2010}. The reduced-path picture extends to geometric rough paths \cite{BGLY2016}, and reconstruction results show how the reduced trajectory is encoded by the signature \cite{Geng2017}. Write
\[
\Sl(V)=\set{S(\eta):\eta\text{ is a continuous bounded-variation }V\text{-valued path}}
\]
for the rectifiable signature group. The problem is to classify the elements of $\Sl(V)$ with entire logarithm and then identify their tree-reduced representatives.

\subsection{Why conjugacy is the natural normal form}\label{sec:conjugacy-normal-form}

A hypothesis imposed only on the signature of the whole path cannot force literal straightness. For based paths, let $*$ denote translated concatenation, let $\overleftarrow{Y}$ be the based reverse of $Y$, and write $\lambda_u(t)=tu$ for the line segment with increment $u$. Geometrically, conjugation attaches a path and its reverse on the two sides of a middle path. If $a=S(Y)$ and $l=\log S(X)$, then
\[
\log S(Y*X*\overleftarrow{Y})=ala^{-1}.
\]
By \cref{prop:conjugacy-invariance}, infinite homogeneous radius is invariant under this operation. Conjugacy is therefore part of the symmetry of a whole-path radius condition.

For linearly independent $v,w\in V$, the three-segment path
\[
\beta=\lambda_w*\lambda_v*\lambda_{-w}
\]
is nonlinear and tree-reduced, but
\[
\log S(\beta)=\e^w v\e^{-w},
\qquad
\norm{\bigl(\log S(\beta)\bigr)_n}_{\pi,n}
 \leq \norm{v}\,\frac{(2\norm{w})^{n-1}}{(n-1)!}.
\]
Thus the global normal form must allow conjugacy to a line signature:
\[
S(\gamma)=a\e^v a^{-1},
\qquad v=\gamma_T-\gamma_0,
\qquad a\in\Sl(V).
\]

\subsection{Main results}\label{sec:main-results}

\begin{theorem}[Classification of entire logarithmic signatures]\label{thm:main}
Let $V$ be a finite-dimensional real normed space, let $\gamma:[0,T]\to V$ be continuous and of bounded variation, and put
\[
\overline\gamma_t=\gamma_t-\gamma_0,
\qquad
 g=S(\gamma),\qquad l=\log g,\qquad v=\gamma_T-\gamma_0.
\]
Assume $R(l)=\infty$.
\begin{enumerate}[label=\textup{(\roman*)}]
\item If $v=0$, then $g=1$; equivalently, $\gamma$ is tree-like.
\item If $v\neq0$, there is $t_*\in[0,T]$ such that, with
\[
\alpha=\overline\gamma|_{[0,t_*]},
\qquad a=S(\alpha),
\]
one has
\begin{equation}\label{eq:main-factorisation}
S(\gamma)=a\e^v a^{-1}.
\end{equation}
\end{enumerate}
Conversely, if $a$ is a bounded-variation signature, then $a\e^v a^{-1}$ has an entire logarithmic signature for every $v\in V$.
\end{theorem}

Equivalently, for
\[
\mathcal E(V)=\set{g\in\Sl(V):R(\log g)=\infty},
\]
\begin{equation}\label{eq:entire-locus}
\mathcal E(V)=\set{a\e^v a^{-1}:a\in\Sl(V),\ v\in V},
\qquad
\set{g\in\mathcal E(V):g_1=0}=\set{1}.
\end{equation}
The first level determines $v$, while the higher levels in the entire locus are generated by conjugation within the rectifiable signature group. The factor $a$ in the nonzero-increment case may be chosen as the signature of an actual prefix of the given centred path.
Thus the analytic condition on all homogeneous free-Lie components is equivalent, within the rectifiable signature group, to a conjugacy normal form. Equation~\eqref{eq:entire-locus} also separates the zero-increment case: the identity is the only entire signature with vanishing first level.

Let $\sim_{\mathrm{rep}}$ denote weak increasing reparametrisation as defined in \cref{sec:path-conventions}, and call a path tree-reduced when, after waiting intervals are removed, it has no nonconstant subpath of trivial signature. After an auxiliary Euclidean norm is chosen, the Le Donne--Z\"ust construction identifies $\Sl(V)$ with a complete metric tree \cite{LeDonneZust2021}. In this tree, the gate of a point to a closed convex set is its metric projection, and a minimum-displacement axis is the locus on which an isometry attains its minimum displacement.

\begin{corollary}[Modified Lyons--Sidorova conjecture]\label{cor:modified-conjecture}
Let $V$ be a finite-dimensional real normed space, let $\gamma:[0,T]\to V$ be a continuous tree-reduced bounded-variation path, and let $v=\gamma_T-\gamma_0$. Then
\[
R\bigl(\log S(\gamma)\bigr)=\infty
\]
if and only if there is a based prefix $\beta$ of $\overline\gamma=\gamma-\gamma_0$ such that
\begin{equation}\label{eq:path-level-main}
\overline\gamma
 \sim_{\mathrm{rep}}
\beta*\lambda_v*\overleftarrow{\beta}.
\end{equation}
When $v=0$, the tree-reduced path is constant. When $v\neq0$, $S(\beta)$ is the gate from the identity to the unique minimum-displacement axis of left multiplication by $S(\gamma)$ in the signature tree. This prefix is unique up to waiting intervals and weak increasing reparametrisation.
\end{corollary}

The analytic input used to obtain matrix isospectrality is the following finite-dimensional statement. For $A\in M_q(\C)$, let $\sigma_k(A)$ be the $k$th elementary symmetric function of its eigenvalues, counted with algebraic multiplicity, and write $I=I_q$ when the matrix size is clear.

\begin{theorem}[Finite-dimensional spectral-growth theorem]\label{thm:spectral-growth}
Let $q\geq1$ and let $A:\C\to M_q(\C)$ be entire, with $A(0)=0$, and suppose
\begin{equation}\label{eq:exp-growth}
\norm{\e^{A(z)}}\leq C\e^{\tau\abs{z}}
\qquad (z\in\C)
\end{equation}
for constants $C\geq1$ and $\tau\geq0$, where $\norm{\cdot}$ is an operator norm. Then $z\mapsto\sigma_k(A(z))$ is a polynomial of degree at most $k$ for $1\leq k\leq q$.

If, in addition, $A(z)=zH+O(z^2)$ as $z\to0$ for some $H\in M_q(\C)$, then
\begin{equation}\label{eq:charpoly-linear}
\det(\zeta I-A(z))=\det(\zeta I-zH)
\end{equation}
for all $z,\zeta\in\C$. Consequently,
\begin{equation}\label{eq:trace-powers}
\tr(A(z)^p)=z^p\tr(H^p)
\qquad (p\geq1).
\end{equation}
\end{theorem}

The theorem is stated in a self-contained matrix form, but its role in this paper is as a purpose-built analytic lemma for the signature classification. It controls the characteristic invariants of $A(z)$ without choosing global holomorphic eigenvalue labels or scalar logarithms of the eigenvalues of $\e^{A(z)}$. The spectral Schwarz estimate is standard \cite[Theorem~1]{Vesentini2008}; the proof adds a half-plane normalisation that turns the bound on $\e^{A(z)}$ into polynomial degree bounds for the characteristic coefficients. Applied to
\[
A_{\Phi}(z)=\sum_{n\geq1}z^n\Phi^{(n)}(l_n),
\qquad \Phi:V_{\C}\to M_q(\C),
\]
it yields
\[
\det\!\bigl(\zeta I_q-A_{\Phi}(z)\bigr)
 =\det\!\bigl(\zeta I_q-z\Phi(v)\bigr).
\]
This identity is used for loop rigidity and for the resonant identities in the nonzero-increment case. No broader spectral-theoretic classification is asserted.

\begin{remark}[The one-dimensional case]\label{rem:one-dimensional}
If $\dim V=1$, atomless product-measure symmetry gives
\[
S_n(\gamma)=\frac{(\gamma_T-\gamma_0)^{\otimes n}}{n!},
\]
so $S(\gamma)=\e^{\gamma_T-\gamma_0}$. The remaining argument concerns $\dim V\geq2$.
\end{remark}

\subsection{Comparison with Boedihardjo--Geng--Wang}\label{sec:prior-work}

All formal theorem and priority comparisons in this subsection are made with Boedihardjo--Geng--Wang \cite{BGW2025}, specifically arXiv:2506.18207v1 of 22 June 2025. A conjugacy-corrected formulation also appeared in Boedihardjo's 2020 DataSig lecture slides \cite[slide~38]{Boedihardjo2020Slides}; the slides are cited only as historical context, and no mathematical comparison below depends on that non-journal source.

\begin{table}[H]
\centering
\caption{Comparison with the results and formulation in Boedihardjo--Geng--Wang \cite{BGW2025}.}
\label{tab:bgw-comparison}
\footnotesize
\renewcommand{\arraystretch}{1.12}
\begin{tabularx}{\textwidth}{@{}p{0.265\textwidth}XX@{}}
\toprule
Issue & Boedihardjo--Geng--Wang & Present paper \\
\midrule
Modified formulation
& Stated as the Modified LS Conjecture.
& Uses the same conjugacy-corrected formulation; it is not claimed here as new. \\
Conjugacy invariance and sufficiency
& Proposition~2.6 proves invariance and the sufficiency direction.
& Reproved in \cref{prop:conjugacy-invariance} and \cref{prop:converse} for self-containment. \\
All-subinterval theorem
& Theorem~4.1 assumes infinite radius on every subinterval and obtains straight-line rigidity.
& Assumes infinite radius only for the signature of the whole path. \\
Integral-frequency identities
& Theorem~6.5 gives logarithmic-signature identities under a nondegeneracy condition.
& \Cref{prop:resonance} gives a full decorated-signature identity for every finite integral-frequency set, without that condition. \\
Zero increment
& The cited results do not give a general classification of loops.
& \Cref{thm:loop-rigidity} proves that an entire bounded-variation loop has signature $1$. \\
Global closure
& No passage from one whole-path hypothesis to the conjugacy factorisation is asserted in the cited results.
& \Cref{sec:tree-factorisation,sec:common-prefix,sec:reconstruction} use tree fixed points, compactness, and Fourier reconstruction to obtain the factorisation. \\
Actual-prefix factorisation
& Not supplied by the cited results.
& \Cref{thm:main} chooses the conjugating signature from a prefix of the given centred path. \\
Gate prefix and path statement
& Not supplied by the cited results.
& \Cref{cor:modified-conjecture,thm:reduction-lemma} select the gate prefix for a tree-reduced representative. \\
\bottomrule
\end{tabularx}
\end{table}

The two arguments share finite-dimensional path developments, complex exponential one-forms, and root-chain constructions. In particular, the simple-root chain in \cite{BGW2025} is a conceptual predecessor of the adjacent matrix-unit chain used here. The difference at the resonance stage is the output: universal matrix isospectrality is converted into an identity for the full decorated signature for arbitrary finite rational frequency sets. On the common continuous bounded-variation domain, its integral-frequency specialisation implies the nondegenerate identities of \cite[Theorem~6.5]{BGW2025}. The subsequent tree--compactness--Fourier argument supplies the whole-path necessity implication and the two prefix conclusions recorded in \cref{tab:bgw-comparison}. The proof does not invoke the theorems of \cite{BGW2025}, although it uses the shared development and root-chain ideas described above.

\subsection{Proof architecture}\label{sec:broader-context}

The proof has four steps.
\begin{enumerate}[label=\textup{\arabic*.},leftmargin=2.2em]
\item \Cref{sec:matrix-isospectrality,sec:zero-increment} apply the spectral-growth theorem to finite-dimensional evaluations of the logarithmic signature, obtaining exact isospectrality and the loop theorem.
\item For $v\neq0$, \cref{sec:resonance} uses resonant nearest-neighbour developments to obtain cyclic identities for every finite rational frequency set.
\item \Cref{sec:tree-factorisation,sec:common-prefix} convert each cyclic identity into a prefix factorisation in the complete signature tree and use compactness to choose one prefix time for all rational frequencies.
\item \Cref{sec:reconstruction} identifies decorated coefficients as Fourier transforms of finite Stieltjes measures and reconstructs every ordinary signature coefficient. The axis--gate reduction then gives the path statement for tree-reduced representatives.
\end{enumerate}
The four steps place the radius problem at the intersection of three structures. First, $l$ is a degreewise free Lie series. Although $R(l)=\infty$ is not free holomorphicity in the sense of \cite{AglerMcCarthy2015}, it makes every finite-dimensional evaluation entire, and the matrix developments form a separating family. Second, the Le Donne--Z\"ust signature space is a complete real tree, so fixed-point, axis, and gate geometry convert algebraic factorisations into path statements \cite{LeDonneZust2021,CullerMorgan1987}. Third, decorated signature coefficients are Fourier transforms of finite signed Stieltjes measures on ordered simplices; Fourier uniqueness and polynomial moments recover the ordinary tensor coefficients. \Cref{fig:proof-dependencies} records the dependencies, and the external inputs and their hypotheses are listed in \cref{sec:external}.

\begin{figure}[ht]
\centering
\begin{tikzpicture}[
  proofbox/.style={draw=black!55, rounded corners=2pt, align=center,
    text width=3.20cm, minimum height=1.20cm, inner sep=3pt,
    fill=black!2, font=\footnotesize},
  proofarrow/.style={-{Latex[length=2mm]}, line width=0.55pt}
]
\node[proofbox] (entire) at (0,0) {Entire logarithmic\\signature};
\node[proofbox] (spectral) at (3.9,0) {Exact matrix\\isospectrality};
\node[proofbox] (resonance) at (7.8,0) {Resonant Fourier\\identities};
\node[proofbox] (prefix) at (11.7,0) {Common prefix\\for all rational\\frequencies};

\node[proofbox] (loop) at (0,-2.10) {Zero increment:\\loop rigidity\\$S(\gamma)=1$};
\node[proofbox] (axis) at (3.9,-2.10) {\mbox{Invariant-axis}\\reduction for\\\mbox{tree-reduced} paths};
\node[proofbox] (factor) at (7.8,-2.10) {Signature\\factorisation\\$S(\gamma)=a\e^v a^{-1}$};
\node[proofbox] (fourier) at (11.7,-2.10) {Stieltjes--Fourier\\reconstruction};

\draw[proofarrow] (entire) -- (spectral);
\draw[proofarrow] (spectral) -- (resonance);
\draw[proofarrow] (resonance) -- (prefix);
\draw[proofarrow] (spectral.south west) to[bend right=18]
  node[left, font=\scriptsize] {$v=0$} (loop.north east);
\draw[proofarrow] (prefix) -- (fourier);
\draw[proofarrow] (fourier) -- (factor);
\draw[proofarrow] (factor) -- (axis);
\end{tikzpicture}
\caption{Dependency structure of the proof. The lower-left branch treats zero increment; the lower row, read from right to left, treats nonzero increment.}
\label{fig:proof-dependencies}
\end{figure}
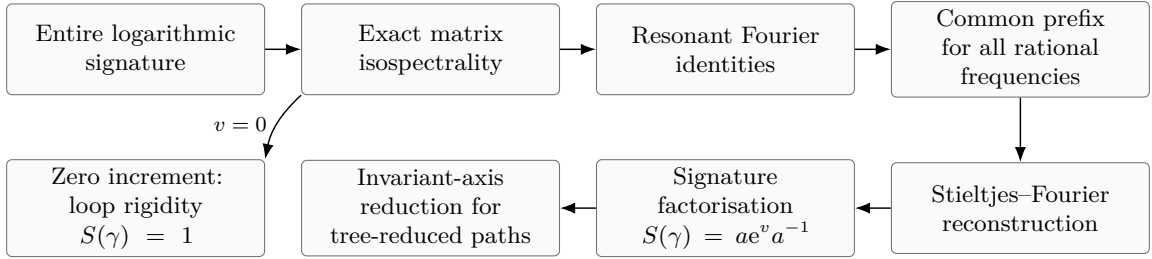

\section{Notation and preliminary material}

\subsection{Norms, tensor algebra, and signatures}

Let $V$ be a finite-dimensional real normed vector space.
Write $V_{\C}=V\otimes_{\R}\C$ for its complexification and fix the complexification norm
\[
\norm{z}_{V_{\C}}
 = \sup_{\substack{\lambda\in V^*\\ \norm{\lambda}\leq 1}}
   \abs{\lambda_{\C}(z)},
\]
where $\lambda_{\C}$ denotes the complex-linear extension of $\lambda$.
This norm restricts to the original norm on $V$.
On $V_{\C}^{\otimes n}$ we use the associated projective tensor norm $\norm{\cdot}_{\pi,n}$.
For each $q\geq1$, $M_q(\C)$ carries the operator norm induced by the Euclidean norm on $\C^q$.

Write
\[
\Tens((V))=\prod_{n=0}^{\infty} V^{\otimes n},
\qquad
\Tens((V_{\C}))=\prod_{n=0}^{\infty} V_{\C}^{\otimes n},
\]
for the completed tensor algebras over $V$ and $V_{\C}$, both with graded concatenation product. If $a=\sum_{n\geq0}a_n$ is a graded series, $a_n$ denotes its homogeneous degree-$n$ component. We write $\langle\cdot,\cdot\rangle$ for the canonical pairing between a tensor power of the dual and the corresponding tensor power.
For a series $l=\sum_{n\geq1}l_n\in\Tens((V_{\C}))$ with zero constant term, define
\[
R(l)=\left(\limsup_{n\to\infty}\norm{l_n}_{\pi,n}^{1/n}\right)^{-1},
\]
with the convention $0^{-1}=\infty$.

\begin{lemma}[Norm independence of infinite radius]\label{lem:norm-independence}
The condition $R(l)=\infty$ is independent of the norm on the finite-dimensional space $V$ and of the choice of any crossnorm family $(\alpha_n)$ satisfying
\[
\norm{u}_{\varepsilon,n}\leq\norm{u}_{\alpha,n}\leq\norm{u}_{\pi,n}
\qquad (u\in V_{\C}^{\otimes n}),
\]
where $\varepsilon_n$ and $\pi_n$ denote the injective and projective norms. This includes the injective and projective families and, after choosing an equivalent Euclidean norm on $V$, the Hilbertian tensor norms.
\end{lemma}

\begin{proof}
If two norms on $V_{\C}$ satisfy
\[
c\norm{z}_1\leq \norm{z}_2\leq C\norm{z}_1,
\]
their induced projective tensor norms satisfy
\[
c^n\norm{u}^{(1)}_{\pi,n}
 \leq \norm{u}^{(2)}_{\pi,n}
 \leq C^n\norm{u}^{(1)}_{\pi,n}.
\]
For the comparison of crossnorms, choose a basis $e_1,\ldots,e_d$ of $V_{\C}$ and its dual basis $e_1^*,\ldots,e_d^*$. Expansion in the associated product basis gives a constant $K\geq1$, depending only on the basis, such that
\[
\norm{u}_{\pi,n}\leq K^n\norm{u}_{\varepsilon,n},
\qquad
K=d\Bigl(\max_j\norm{e_j}\Bigr)
     \Bigl(\max_j\norm{e_j^*}\Bigr)
\]
is one admissible choice. Hence any two such families $(\alpha_n)$ and $(\beta_n)$ satisfy
\[
K^{-n}\norm{u}_{\alpha,n}
 \leq \norm{u}_{\beta,n}
 \leq K^n\norm{u}_{\alpha,n}.
\]
Their homogeneous root growth rates therefore differ by at most fixed multiplicative factors, and one vanishes if and only if the other does.
\end{proof}

Whenever an imported result is formulated for a Euclidean state space, we choose an auxiliary Euclidean norm on $V$. In finite dimensions, equivalent norms define the same continuous bounded-variation and rectifiable paths. Equality of signatures, tree-likeness, tree-like equivalence, and tree-reducedness are algebraic or path-theoretic notions and are therefore unchanged. The norm-independence of the signature-tree arcs and gates used later is proved explicitly in \cref{prop:prefix-geodesic,thm:reduction-lemma}.

For a continuous bounded-variation path $\gamma:[0,T]\to V$, set $S_0(\gamma)=1$ and
\[
S_n(\gamma)=
\int_{0<t_1<\cdots<t_n<T}
\dd\gamma_{t_1}\otimes\cdots\otimes\dd\gamma_{t_n}
\qquad(n\geq1),
\]
so that
\[
S(\gamma)=\sum_{n\geq0}S_n(\gamma).
\]
The integrals may be interpreted as tensor-valued Riemann--Stieltjes integrals, or equivalently through finite signed product measures induced by scalar bounded-variation coordinates.
Write
\[
L(\gamma)=\Var(\gamma;[0,T]);
\]
when the parameter interval is clear, we also write $\Var(\gamma)$ for this quantity. Then
\begin{equation}\label{eq:factorial-signature}
\norm{S_n(\gamma)}_{\pi,n}
 \leq \frac{L(\gamma)^n}{n!}
\qquad(n\geq0).
\end{equation}

The signature is group-like, and its formal logarithm
\[
l=\log S(\gamma)=\sum_{n\geq1}l_n
\]
is primitive. We call the logarithmic signature $l$ \emph{entire} when $R(l)=\infty$. Its first level is
\[
v=l_1=S_1(\gamma)=\gamma_T-\gamma_0.
\]
For $z\in\C$, define the homogeneous dilation on the complexified tensor algebra by
\[
\delta_z:\Tens((V_{\C}))\longrightarrow\Tens((V_{\C})),
\qquad
\delta_z\!\left(\sum_{n\geq0}a_n\right)=\sum_{n\geq0}z^n a_n,
\]
and regard $\Tens((V))$ as its real subalgebra. Signatures are translation invariant, and paths will normally be translated so that $\gamma_0=0$.

\subsection{Path conventions, conjugacy, and reducedness}\label{sec:path-conventions}

For continuous paths $\eta:[0,A]\to V$ and $\widetilde\eta:[0,B]\to V$, write
\[
\eta\sim_{\mathrm{rep}}\widetilde\eta
\]
when there are continuous nondecreasing surjections $\sigma:[0,1]\to[0,A]$ and $\widetilde\sigma:[0,1]\to[0,B]$ such that
$\eta\circ\sigma=\widetilde\eta\circ\widetilde\sigma$.
This permits changes of speed and constant waiting intervals, but not reversal.
A path is \emph{based} if its initial point is the origin. All concatenation factors below are based; $*$ denotes the usual translated concatenation, and the based reverse of $\alpha:[0,A]\to V$ is
\[
\overleftarrow{\alpha}_s=\alpha_{A-s}-\alpha_A,
\qquad 0\leq s\leq A.
\]
We write $\lambda_v(t)=tv$, $0\leq t\leq1$.

A continuous bounded-variation path is \emph{tree-reduced} when, after waiting intervals are suppressed, no nonconstant subpath has trivial signature. \Cref{prop:reducedness,prop:weak-reparametrisation,prop:prefix-geodesic} relate this formulation to Hambly--Lyons reducedness, weak reparametrisation, and prefix-signature geodesics.

\begin{definition}[Path and signature conjugacy]\label{def:two-conjugacies}
A continuous path $\gamma:[0,T]\to V$ is \emph{path-conjugate to the line $\lambda_v$} if $v=\gamma_T-\gamma_0$ and there is a based bounded-variation path $\alpha$ such that
\[
\gamma-\gamma_0
 \sim_{\mathrm{rep}}
\alpha*\lambda_v*\overleftarrow{\alpha}.
\]
Its signature is \emph{conjugate to the line signature} if
\[
S(\gamma)=a\e^v a^{-1}
\]
for some bounded-variation signature $a$.
\end{definition}

Literal path conjugacy implies signature conjugacy by Chen's identity. For tree-reduced paths, \cref{thm:reduction-lemma} gives the converse at path level.

\subsection{The conjugacy-invariance estimate}\label{subsec:conjugacy-estimate}

\begin{proposition}[Conjugacy invariance of infinite radius]\label{prop:conjugacy-invariance}
Let $X$ and $Y$ be based continuous bounded-variation paths. Then
\[
R\bigl(\log S(X)\bigr)=\infty
\quad\Longleftrightarrow\quad
R\bigl(\log S(Y*X*\overleftarrow{Y})\bigr)=\infty.
\]
More precisely, if $l=\log S(X)$ and $a=S(Y)$, then
\[
\log S(Y*X*\overleftarrow{Y})=a l a^{-1}.
\]
\end{proposition}

The proof uses only the factorial bounds for bounded-variation signatures and submultiplicativity of the projective tensor norm.

\begin{proof}
Chen's identity gives
\[
S(Y*X*\overleftarrow{Y})=a\e^l a^{-1}=\e^{a l a^{-1}}.
\]
Let $A=L(Y)$ and write $a_p$ and $(a^{-1})_q$ for the homogeneous components of $a$ and $a^{-1}$.
They satisfy
\[
\norm{a_p}_{\pi,p}\leq \frac{A^p}{p!},
\qquad
\norm{(a^{-1})_q}_{\pi,q}\leq \frac{A^q}{q!}.
\]
If $l$ is entire, then for every $\varepsilon>0$ there is $M_\varepsilon<\infty$ such that
$\norm{l_k}_{\pi,k}\leq M_\varepsilon\varepsilon^k$ for all $k\geq1$.
Consequently,
\[
\norm{(a l a^{-1})_n}_{\pi,n}
\leq
M_\varepsilon
\sum_{k=1}^{n}
\varepsilon^k\frac{(2A)^{n-k}}{(n-k)!}.
\]
Writing $j=n-k$ gives the explicit estimate
\[
\begin{aligned}
\sum_{k=1}^{n}
\varepsilon^k\frac{(2A)^{n-k}}{(n-k)!}
&=\varepsilon^n
  \sum_{j=0}^{n-1}\frac{(2A/\varepsilon)^j}{j!}\\
&\leq \e^{2A/\varepsilon}\varepsilon^n.
\end{aligned}
\]
Hence
\[
\limsup_{n\to\infty}
\norm{(a l a^{-1})_n}_{\pi,n}^{1/n}
\leq \varepsilon.
\]
Since $\varepsilon>0$ is arbitrary, $a l a^{-1}$ has infinite homogeneous radius.
The converse follows by conjugating with $a^{-1}=S(\overleftarrow{Y})$.
\end{proof}

\subsection{External inputs}\label{sec:external}

The proof imports the following results; all other ingredients are established below or used in their standard finite-dimensional or bounded-variation form.

\begin{enumerate}
\item \textbf{Vesentini subharmonicity.}
For a holomorphic map into a complex unital Banach algebra, the logarithm of the spectral radius is an extended-real-valued subharmonic function, taking the value $-\infty$ at quasinilpotent points \cite[pp.~427--429]{Vesentini1968}; see also \cite[Theorem~6.4.2]{Ransford1995}.

\item \textbf{Holomorphic and rational spectral mapping.}
If $f$ is holomorphic on a neighbourhood of the spectrum of a matrix or Banach-algebra element, then $\spec f(A)=f(\spec A)$. The same conclusion holds for a rational function whose poles avoid $\spec A$; see \cite[Chapter~VII, \S4]{Conway1990}.

\item \textbf{Primitive elements and free Lie polynomials.}
In each finite homogeneous degree, the primitive elements of the tensor Hopf algebra are precisely the free Lie elements; see Friedrichs' theorem as presented in \cite{Reutenauer1993}. Thus every homogeneous component of a logarithmic signature is a finite free Lie polynomial.

\item \textbf{Signature uniqueness and reduced representatives.}
Trivial signature is equivalent to tree-likeness, equality of signatures is equivalent to tree-like equivalence, and each bounded-variation equivalence class has a unique reduced representative up to translation and increasing reparametrisation \cite{HamblyLyons2010}. In the metric on $G_{\mathrm{p.r.c.}}^{(*)}$ from \cite[Definition~2.1]{BGLY2016}, Boedihardjo--Geng--Lyons--Yang prove existence and uniqueness, up to reparametrisation, of the injective reduction of every continuous finite-$p$-variation $G_{\mathrm{p.r.c.}}^{(*)}$-valued path \cite[Lemma~4.6]{BGLY2016}; for $p=1$, its first-level projection is the Hambly--Lyons reduced representative \cite[Remark~4.1]{BGLY2016}. \Cref{prop:reducedness} verifies the continuity and finite-$1$-variation hypotheses in that metric before applying the result, and combines it with a direct excision argument to characterise tree-reduced bounded-variation paths by the absence of nonconstant trivial-signature subpaths. Weak reparametrisation is treated in \cref{prop:weak-reparametrisation}, and prefix geodesics in \cref{prop:prefix-geodesic}.

\item \textbf{Metric signature tree.}
For a finite-dimensional Euclidean space $E$, let $G_N(E)$ be the free step-$N$ nilpotent group, let $d_N$ be its Carnot--Carath\'eodory metric, and let $1_N$ be its identity. Le Donne--Z\"ust do not use the unrestricted algebraic inverse limit of all compatible truncations. Their inverse limit consists of the compatible sequences $(g_N)_{N\geq1}$ with $g_N\in G_N(E)$ satisfying the finite-distance condition
\[
\sup_{N\geq1}d_N(1_N,g_N)<\infty.
\]
Completeness of precisely this metric space is \cite[Lemma~2.2(2)]{LeDonneZust2021}; its coordinatewise group structure and left invariance are \cite[Lemma~3.1]{LeDonneZust2021}. The natural group isomorphism with rectifiable signatures and the complete metric-tree conclusion are \cite[Theorem~4.4]{LeDonneZust2021}. Thus no additional completion points are introduced. The metric-tree argument there uses Hambly--Lyons signature uniqueness.

\item \textbf{Fourier uniqueness.}
A finite complex Borel measure on $\R^r$ is determined by its Fourier transform \cite{Rudin1990}.
\end{enumerate}

\section{Exact matrix isospectrality}\label{sec:matrix-isospectrality}

\subsection{Evaluation of the formal exponential}

Fix $q\geq1$, write $I=I_q$ for the identity in $M_q(\C)$, and let $\Phi:V_{\C}\to M_q(\C)$ be complex linear, with operator norm $\norm{\Phi}$ relative to the norms fixed above. Its multiplicative extension on $V_{\C}^{\otimes n}$ is denoted by $\Phi^{(n)}$ for $n\geq1$, and $\Phi^{(0)}(1)=I$.
By the defining property of the projective tensor norm,
\begin{equation}\label{eq:phi-bound}
\norm{\Phi^{(n)}(a_n)}
 \leq \norm{\Phi}^n\norm{a_n}_{\pi,n}.
\end{equation}
The notation $\Phi(\sum_{n\geq0}a_n)$ below denotes the convergent series $\sum_{n\geq0}\Phi^{(n)}(a_n)$.

\begin{lemma}[Analytic evaluation lemma]\label{lem:analytic-evaluation}
Let $\gamma:[0,T]\to V$ be continuous and of bounded variation, and put $g=S(\gamma)$ and $l=\log g$.
Assume $R(l)=\infty$.
Define
\[
A_{\Phi}(z)=\sum_{n\geq 1}z^n\Phi^{(n)}(l_n),
\qquad
G_{\Phi}(z)=\sum_{n\geq 0}z^n\Phi^{(n)}(S_n(\gamma)).
\]
Then both functions are entire and
\begin{equation}\label{eq:evaluation-exp}
G_{\Phi}(z)=\e^{A_{\Phi}(z)}
\qquad (z\in\C).
\end{equation}
Moreover,
\begin{equation}\label{eq:G-growth}
\norm{G_{\Phi}(z)}
 \leq \exp\!\bigl(\abs{z}\,\norm{\Phi}\,L(\gamma)\bigr).
\end{equation}
\end{lemma}

\begin{proof}
The series defining $A_{\Phi}$ converges locally uniformly by $R(l)=\infty$ and \cref{eq:phi-bound}.
The factorial estimate \cref{eq:factorial-signature} gives local uniform convergence of $G_{\Phi}$ and directly yields \cref{eq:G-growth}.

The formal identity $g=\e^l$ is homogeneous: for every $n$, the degree-$n$ coefficient of $g$ is a finite sum of products involving only $l_1,\ldots,l_n$.
After applying the multiplicative maps $\Phi^{(k)}$, the Taylor coefficient of degree $n$ in $G_{\Phi}$ is therefore the Taylor coefficient of degree $n$ in $\e^{A_{\Phi}}$.
Both sides are entire matrix-valued functions, so equality of all Taylor coefficients gives \cref{eq:evaluation-exp}.
\end{proof}

\begin{lemma}[Matrix development and the evaluated signature]\label{lem:development-signature}
Let $q\geq1$, write $I=I_q$, let $\beta:[0,T_\beta]\to V$ be continuous and of bounded variation, and let $\Phi:V_{\C}\to M_q(\C)$ be complex linear. For $z\in\C$, let $U^z$ be the solution of the right matrix-development equation
\begin{equation}\label{eq:matrix-development}
\dd U_t^z=U_t^z\,z\Phi(\dd\beta_t),
\qquad U_0^z=I.
\end{equation}
Then
\begin{align}
U_{T_\beta}^z
&=I+\sum_{n\geq1}z^n
  \int_{0<t_1<\cdots<t_n<T_\beta}
  \Phi(\dd\beta_{t_1})\cdots\Phi(\dd\beta_{t_n}) \notag\\
&=\sum_{n\geq0}z^n\Phi^{(n)}\!\left(S_n(\beta)\right).
\label{eq:development-signature}
\end{align}
The series converges absolutely and locally uniformly in $z$, and
\[
\norm{U_{T_\beta}^z}
 \leq \exp\!\bigl(\abs{z}\,\norm{\Phi}\,L(\beta)\bigr).
\]
If $R(\log S(\beta))=\infty$ and
$\log S(\beta)=\sum_{n\geq1}b_n$, then
\begin{equation}\label{eq:development-logarithm}
U_{T_\beta}^z
 =\exp\!\left(\sum_{n\geq1}z^n\Phi^{(n)}(b_n)\right).
\end{equation}
\end{lemma}

\begin{proof}
Picard iteration of \cref{eq:matrix-development} gives the first series in
\cref{eq:development-signature}.
Its degree-$n$ term has norm at most
\[
\frac{\bigl(\abs{z}\,\norm{\Phi}\,L(\beta)\bigr)^n}{n!},
\]
so the series converges absolutely and locally uniformly and satisfies the displayed exponential bound.
Because the extension of $\Phi$ is multiplicative and the right-development convention orders the matrix factors in the same order as the tensor factors,
\[
\int_{0<t_1<\cdots<t_n<T_\beta}
\Phi(\dd\beta_{t_1})\cdots\Phi(\dd\beta_{t_n})
 =\Phi^{(n)}\!\left(S_n(\beta)\right).
\]
This proves \cref{eq:development-signature}; the same Picard estimate also gives uniqueness of the solution.
Under the entire-logarithm hypothesis, \cref{lem:analytic-evaluation} identifies the last series in \cref{eq:development-signature} with the exponential in \cref{eq:development-logarithm}.
\end{proof}

\subsection{A spectral Schwarz lemma}

For $q\geq1$ and $A\in M_q(\C)$, write $r(A)$ for its spectral radius.
Write
\[
\mathbb D=\set{z\in\C:\abs{z}<1},
\qquad
\overline{\mathbb D}_{\rho}=\set{z\in\C:\abs{z}\leq\rho}
\quad(\rho>0).
\]

\begin{lemma}[Extended-valued maximum principle on a disc]\label{lem:extended-subharmonic-maximum}
Let $\Omega\subset\C$ be open and contain the closed disc $\overline{\mathbb{D}}_{\rho}$, and let $u:\Omega\to[-\infty,\infty)$ be subharmonic. Assume that $u$ is not identically $-\infty$ on the connected component of $\Omega$ containing the disc and that $u$ is finite at some point of $\set{\abs{z}=\rho}$. Then
\[
\sup_{\abs{z}\leq\rho}u(z)
 \leq \sup_{\abs{z}=\rho}u(z).
\]
\end{lemma}

\begin{proof}
Put $M=\sup_{\abs{z}=\rho}u(z)$. Upper semicontinuity on the compact boundary circle gives $M<\infty$. Fix $\varepsilon>0$. Upper semicontinuity and compactness give a relative neighbourhood $U$ of the boundary circle in $\overline{\mathbb{D}}_{\rho}$ on which $u<M+\varepsilon$. If $u>M+\varepsilon$ somewhere in the complement of $U$, then upper semicontinuity on that compact complement makes $u$ attain a finite maximum greater than $M+\varepsilon$ at an interior point of the disc. The local maximum principle for subharmonic functions would force $u$ to be constant on the relevant connected component, contradicting the bound on $U$. Hence $u\leq M+\varepsilon$ on the closed disc. Letting $\varepsilon\downarrow0$ proves the claim. This is the standard extended-valued maximum principle; see \cite[Section~2.3]{Ransford1995}.
\end{proof}

The following result is the finite-dimensional matrix case of the standard spectral Schwarz theorem for holomorphic maps into a unital Banach algebra; see \cite[Theorem~1]{Vesentini2008}. We include the proof for completeness, to fix the normalisation used below and to record explicitly the treatment of quasinilpotent values in the extended-valued subharmonic argument.

\begin{lemma}[Spectral Schwarz lemma]\label{lem:spectral-schwarz}
Let $q\geq1$ and let $F:\mathbb{D}\to M_q(\C)$ be holomorphic, with $F(0)=0$ and $r(F(z))<1$ for every $z\in\mathbb{D}$.
Then
\[
r(F(z))\leq \abs{z}
\qquad (z\in\mathbb{D}).
\]
\end{lemma}

\begin{proof}
Set $H(z)=F(z)/z$ for $z\neq 0$ and $H(0)=F'(0)$.
This is holomorphic.
Fix $0<\rho<1$.
On $\abs{z}=\rho$,
\[
r(H(z))=\frac{r(F(z))}{\rho}<\frac{1}{\rho}.
\]
In finite dimensions the spectral radius is continuous as a function of the matrix entries.
Hence $z\mapsto r(H(z))$ attains on the boundary circle a maximum
\[
m_\rho:=\max_{\abs{z}=\rho}r(H(z))<\rho^{-1};
\]
the strict inequality follows from the pointwise strict bound and compactness.
If $m_\rho=0$, every $H(z)$ on the circle is quasinilpotent.
The non-leading coefficients of $\det(\zeta I-H(z))$ are holomorphic functions of $z$ and vanish on that circle, hence vanish identically on the disc.
Thus $H(z)$ is quasinilpotent throughout the disc and the conclusion is immediate.
Assume $m_\rho>0$.
By Vesentini's theorem \cite[pp.~427--429]{Vesentini1968}; see also \cite[Theorem~6.4.2]{Ransford1995},
\[
u(z)=\log r(H(z))
\]
is subharmonic in the extended-valued sense; values equal to $-\infty$ at quasinilpotent points are allowed. Since $m_\rho>0$, the function is not identically $-\infty$ and is finite at a boundary point. By \cref{lem:extended-subharmonic-maximum},
\[
\sup_{\abs{z}\leq\rho}u(z)
 \leq \sup_{\abs{z}=\rho}u(z)
 =\log m_\rho<-\log\rho.
\]
Therefore
\[
u(z)\leq \log m_\rho< -\log\rho
\qquad (\abs{z}\leq\rho),
\]
and hence $r(H(z))<\rho^{-1}$ on the closed disc.
For fixed $z\in\mathbb{D}$, let $\rho\uparrow1$ through values with $\abs{z}<\rho$.
Then $r(H(z))\leq1$ and $r(F(z))=\abs{z}\,r(H(z))\leq\abs{z}$; the case $z=0$ follows directly from $F(0)=0$.
\end{proof}

\subsection{Polynomial spectral invariants}

For $A\in M_q(\C)$, let $\sigma_k(A)$ be the $k$th elementary symmetric polynomial of the eigenvalues, counted with algebraic multiplicity:
\[
\det(\zeta I-A)
 = \zeta^q-\sigma_1(A)\zeta^{q-1}+\cdots+(-1)^q\sigma_q(A).
\]

We now prove \cref{thm:spectral-growth}. The spectral Schwarz lemma itself is standard; the additional step needed here is to upgrade the one-sided bound on $\Re\spec A(z)$ supplied by \cref{eq:exp-growth} to a uniform modulus bound on the spectrum. A rational map sends the controlling half-plane to the unit disc, where \cref{lem:spectral-schwarz} applies; the resulting linear spectral-modulus bound and Cauchy's estimates then force the degree-$k$ polynomial rigidity of the characteristic coefficients. The argument uses spectral sets and symmetric functions only, without choosing eigenvalue branches or scalar logarithms.

\begin{proof}[Proof of \cref{thm:spectral-growth}]
By the holomorphic spectral mapping theorem \cite[Chapter~VII, \S4]{Conway1990},
\[
r(\e^{A(z)})
 = \exp\!\left(\max\set{\Re\lambda:\lambda\in\spec A(z)}\right).
\]
Since $r(\e^{A(z)})\leq\norm{\e^{A(z)}}$, every $\lambda\in\spec A(z)$ satisfies
\begin{equation}\label{eq:halfplane-growth}
\Re\lambda\leq\tau\abs{z}+\log C.
\end{equation}

Fix $R>0$ and set
\[
M_R=\tau R+\log C+1,
\qquad
F_R(w)=\frac{A(Rw)}{M_R}
\qquad (\abs{w}<1).
\]
By \cref{eq:halfplane-growth}, $\spec F_R(w)$ lies in the open half-plane $\Re\mu<1$.
The scalar map
\[
\phi(\mu)=\frac{\mu}{2-\mu}
\]
fixes $0$ and maps that half-plane biholomorphically onto $\mathbb{D}$. Indeed,
\[
\abs{\phi(\mu)}<1
\quad\Longleftrightarrow\quad
\abs{\mu}<\abs{2-\mu}
\quad\Longleftrightarrow\quad
\Re\mu<1,
\]
and its inverse on $\mathbb{D}$ is $b\mapsto 2b/(1+b)$.
Because $2\notin\spec F_R(w)$, the matrix $2I-F_R(w)$ is invertible.
Its inverse is holomorphic in $w$, and therefore
\[
B_R(w)=\phi(F_R(w))=F_R(w)(2I-F_R(w))^{-1}
\]
is holomorphic.
The exclusion $2\notin\spec F_R(w)$ is precisely the pole-avoidance hypothesis required by the rational spectral mapping theorem \cite[Chapter~VII, \S4]{Conway1990}. Since the only pole of $\phi$ is at $2$, that theorem gives
\[
\spec B_R(w)=\phi(\spec F_R(w)).
\]
In the present matrix setting this can also be checked directly: for $b\neq-1$,
\[
bI-B_R(w)
 =\bigl(2bI-(1+b)F_R(w)\bigr)(2I-F_R(w))^{-1},
\]
so $bI-B_R(w)$ is singular exactly when $2b/(1+b)\in\spec F_R(w)$, equivalently $b=\phi(\mu)$ for some $\mu\in\spec F_R(w)$; for $b=-1$ the left-hand side is $-2(2I-F_R(w))^{-1}$ and is invertible.
Thus $B_R(0)=0$ and $r(B_R(w))<1$.
This direct resolvent identity also confirms that no eigenvalue labelling or logarithm branch is being used: the argument is formulated entirely in terms of spectral sets and symmetric functions of the eigenvalues.
By \cref{lem:spectral-schwarz},
\[
r(B_R(w))\leq\abs{w}.
\]

Suppose $\abs{w}\leq\tfrac12$ and $\mu\in\spec F_R(w)$.
Then
\[
b=\frac{\mu}{2-\mu}\in\spec B_R(w),
\qquad
\abs{b}\leq\tfrac12.
\]
Since $\mu=2b/(1+b)$,
\[
\abs{\mu}
 \leq \frac{2\abs{b}}{1-\abs{b}}
 \leq 2.
\]
Taking $R=2S$, every $z$ with $\abs{z}\leq S$ can be written as $z=Rw$ with $\abs{w}\leq\tfrac12$, and hence
\begin{equation}\label{eq:eigenvalue-linear-bound}
\max_{\lambda\in\spec A(z)}\abs{\lambda}
 \leq 2M_{2S}
 =4\tau S+2\log C+2.
\end{equation}
Thus the original one-sided estimate on $\Re\lambda$ has been upgraded to a uniform modulus bound, linear in the radius of the $z$-disc.
It follows that
\[
\abs{\sigma_k(A(z))}
 \leq \binom{q}{k}
       (4\tau S+2\log C+2)^k
\qquad (\abs{z}\leq S).
\]
Since $A$ is entire and $\sigma_k$ is a polynomial in the matrix entries, the scalar function $z\mapsto\sigma_k(A(z))$ is entire.
Write $\sigma_k(A(z))=\sum_{n\geq 0}c_n z^n$.
Applying Cauchy's coefficient estimate on the circle $\abs{z}=S$ gives, for $n>k$,
\[
\abs{c_n}
 \leq \binom{q}{k}
       \frac{(4\tau S+2\log C+2)^k}{S^n}
 =O(S^{k-n}).
\]
Letting $S\to\infty$ shows $c_n=0$ for every $n>k$.
Thus $\sigma_k(A(z))$ is a polynomial of degree at most $k$.

If $A(z)=zH+O(z^2)$, define
\[
B(z)=
\begin{cases}
A(z)/z, & z\neq0,\\
A'(0), & z=0.
\end{cases}
\]
Then $B$ is entire, $A(z)=zB(z)$, and $B(0)=H$. Since $\sigma_k$ is a homogeneous polynomial of degree $k$ in the matrix entries,
\[
\sigma_k(A(z))=\sigma_k(zB(z))=z^k\sigma_k(B(z)).
\]
The left-hand side is already a polynomial of degree at most $k$, so it must be a constant multiple of $z^k$. Evaluating the remaining factor at $z=0$ gives that constant as $\sigma_k(B(0))=\sigma_k(H)$; hence
\[
\sigma_k(A(z))=z^k\sigma_k(H)
\]
identically.
This proves \cref{eq:charpoly-linear}.
Equality of characteristic polynomials gives equality of the eigenvalue multisets with algebraic multiplicity, and the trace of a matrix power is the corresponding power sum of its eigenvalues.
Hence \cref{eq:trace-powers} follows.
\end{proof}

\subsection{Exact isospectrality for bounded-variation signatures}

\begin{corollary}[Exact matrix isospectrality]\label{thm:exact-isospectrality}
Let $V$ be a finite-dimensional real normed space, let $\gamma:[0,T]\to V$ be continuous and of bounded variation, put $g=S(\gamma)$, let $l=\log g=\sum_{n\geq 1}l_n$, and put $v=l_1$.
Assume $R(l)=\infty$.
For every $q\geq1$ and every complex-linear map $\Phi:V_{\C}\to M_q(\C)$, write $I=I_q$ and define
\[
A_{\Phi}(z)=\sum_{n\geq 1}z^n\Phi^{(n)}(l_n).
\]
Then
\begin{equation}\label{eq:exact-isospectrality}
\det(\zeta I-A_{\Phi}(z))
 =\det(\zeta I-z\Phi(v))
\end{equation}
for all $z,\zeta\in\C$.
Thus $A_{\Phi}(z)$ and $z\Phi(v)$ have the same eigenvalue multiset, counted with algebraic multiplicity, and
\[
\tr\bigl(A_{\Phi}(z)^p\bigr)
 =z^p\tr\bigl(\Phi(v)^p\bigr)
\qquad (p\geq1).
\]
\end{corollary}

\begin{proof}
By \cref{lem:analytic-evaluation},
\[
\e^{A_{\Phi}(z)}=G_{\Phi}(z),
\qquad
\norm{\e^{A_{\Phi}(z)}}
 \leq \exp\!\bigl(\abs{z}\,\norm{\Phi}\,L(\gamma)\bigr).
\]
Also $A_{\Phi}(z)=z\Phi(v)+O(z^2)$.
Apply \cref{thm:spectral-growth} with $C=1$ and $\tau=\norm{\Phi}L(\gamma)$.
\end{proof}

\section{The zero-increment case}\label{sec:zero-increment}

\subsection{Homogeneous skew-Hermitian separation}

Write $\uLie(q)$ for the real Lie algebra of skew-Hermitian $q\times q$ matrices.

\begin{lemma}[Skew-Hermitian separation in one homogeneous degree]\label{lem:skew-separation}
Let $p$ be a nonzero homogeneous real free Lie polynomial of degree $m$ in $d$ generators.
Then there exist $q\geq 1$ and matrices $X_1,\ldots,X_d\in\uLie(q)$ such that
\[
p(X_1,\ldots,X_d)\neq 0.
\]
\end{lemma}

\begin{proof}
Identify the free Lie algebra with its canonical Lie subalgebra of the free associative algebra on the $d$ generators, with bracket $[a,b]=ab-ba$.
Thus the nonzero homogeneous element $p$ is also a nonzero linear combination of associative words of length $m$.
Let $E$ be the finite-dimensional real vector space spanned by all words of length at most $m$ in the $d$ generators.
Let $L_i:E\to E$ be left multiplication by the $i$th generator, truncated to zero on words of length $m$.
Acting on the empty word $1$,
\[
p(L_1,\ldots,L_d)1=p,
\]
so this evaluation is nonzero.
Hence some matrix entry of $p(Z_1,\ldots,Z_d)$, with $Z_j$ generic $q\times q$ matrices and $q=\dim E$, is a nonzero complex polynomial in the matrix entries.

Choose a real basis $H_1,\ldots,H_{q^2}$ of $\uLie(q)$.
The decomposition
\[
M_q(\C)=\uLie(q)\oplus i\uLie(q)
\]
as real vector spaces shows that this real basis is also a complex basis of $M_q(\C)$: complex spanning is immediate, and complex linear independence follows by separating real and imaginary parts in $\uLie(q)$.
Expressing each $Z_j$ as $\sum_k z_{jk}H_k$ turns the chosen matrix entry into a nonzero complex polynomial in the scalar variables $z_{jk}$.
A nonzero complex polynomial cannot vanish on all real tuples $z_{jk}\in\R$.
Thus it is nonzero for some tuple of matrices in $\uLie(q)^d$.

Equivalently, if $e_1,\ldots,e_d$ is a basis of a finite-dimensional real vector space $V$ and the tensor $p(e_1,\ldots,e_d)\in V^{\otimes m}$ is nonzero, the assignment $e_j\mapsto X_j$ extends uniquely to a real-linear development $\rho:V\to\uLie(q)$ and satisfies
\[
\rho^{(m)}\bigl(p(e_1,\ldots,e_d)\bigr)=p(X_1,\ldots,X_d)\neq0.
\]
\end{proof}

\subsection{Rigidity of loops}

\begin{theorem}[Entire-loop rigidity]\label{thm:loop-rigidity}
Let $V$ be a finite-dimensional real normed space and let $\gamma:[0,T]\to V$ be a continuous bounded-variation loop.
If $R(\log S(\gamma))=\infty$, then
\[
S(\gamma)=1.
\]
\end{theorem}

\begin{proof}
Put $l=\log S(\gamma)$.
Since $l_1=0$, \cref{thm:exact-isospectrality} gives, for every $q\geq1$ and every complex-linear development $\rho:V_{\C}\to M_q(\C)$,
\[
\det(\zeta I_q-\rho(\delta_z l))=\zeta^q,
\qquad
\rho(\delta_z l):=\sum_{n\geq1}z^n\rho^{(n)}(l_n).
\]
Suppose $l\neq 0$, and let $m\geq 2$ be minimal with $l_m\neq 0$.
Choose a basis $e_1,\ldots,e_d$ of $V$. By the primitive/free-Lie identification \cite{Reutenauer1993}, $l_m$ is a nonzero homogeneous real free Lie polynomial in these basis generators. By \cref{lem:skew-separation}, the resulting matrices define a real-linear development $\rho:V\to\uLie(q)$, which we extend complex-linearly to $V_{\C}$, such that
\[
H=\rho^{(m)}(l_m)\neq 0.
\]
For each fixed $z\in\C$, the characteristic polynomial of $\rho(\delta_z l)$ is $\zeta^q$, so all of its eigenvalues $\lambda_1(z),\ldots,\lambda_q(z)$, counted with algebraic multiplicity, are zero. The power-sum identity for a matrix, equivalently the second Newton identity, therefore gives
\[
\tr\!\bigl(\rho(\delta_z l)^2\bigr)
 =\sum_{j=1}^q\lambda_j(z)^2
 =0
\qquad (z\in\C).
\]
But
\[
\rho(\delta_z l)=z^mH+O(z^{m+1}),
\]
so the coefficient of $z^{2m}$ is $\tr(H^2)$.
Since $H$ is nonzero and skew-Hermitian,
\[
\tr(H^2)=-\tr(H^*H)<0,
\]
a contradiction.
Hence $l=0$ and $S(\gamma)=1$.
\end{proof}

\section{Rational Fourier decorations and resonant powers}\label{sec:resonance}

Starting from the exact isospectrality of \cref{thm:exact-isospectrality}, we derive a cyclic identity for the full decorated transverse signature associated with each finite symmetric rational frequency set $\Lambda$. Write $i=\sqrt{-1}$. A diagonal matrix records the longitudinal coordinate $x$, nearest-neighbour matrix units record transverse increments, and a gauge transform produces the Fourier decorations $\e^{i\omega x}\dd X$. Rational resonance makes a suitable concatenation power have trivial development, while one matrix coefficient isolates each prescribed decorated word.

Let $\gamma:[0,T]\to V$ be a continuous bounded-variation path, and put
\[
g=S(\gamma),
\qquad
l=\log g,
\qquad
v=\gamma_T-\gamma_0.
\]
Assume throughout this section that
\[
R(l)=\infty,
\qquad
v\neq0,
\qquad
\dim V\geq2.
\]
The one-dimensional case was disposed of in \cref{rem:one-dimensional}.
Choose $\ell\in V^*$ with
\[
\ell(v)=1,
\]
and put $W=\ker\ell$.
After translating $\gamma_0$ to zero, decompose
\begin{equation}\label{eq:path-decomposition}
\gamma_t=x_t v+X_t,
\qquad
x_t=\ell(\gamma_t),
\qquad
X_t\in W.
\end{equation}
Then
\[
x_0=0,\qquad x_T=1,\qquad X_0=X_T=0.
\]
Whenever $\eta\in W^*$ appears below, it is extended to $V^*$ by $\eta(v)=0$.

\subsection{Decorated transverse paths}

Let $\Lambda\subset 2\pi\Q$ be finite and symmetric: $\Lambda=-\Lambda$.
Give $W$ an auxiliary Euclidean norm and let $W_{\C}$ be its complexification.
Define the finite-dimensional real Hilbert space
\begin{equation}\label{eq:E-Lambda}
E_{\Lambda}
 =\set{(z_{\omega})_{\omega\in\Lambda}\in
        \prod_{\omega\in\Lambda}W_{\C}:
        z_{-\omega}=\overline{z_{\omega}}}
\end{equation}
with the norm inherited from the direct sum.
The decorated transverse path is
\begin{equation}\label{eq:decorated-path}
B_t^{\Lambda}(\gamma)
 =\left(\int_0^t \e^{i\omega x_s}\dd X_s\right)_{\omega\in\Lambda}.
\end{equation}
Each coordinate is a vector-valued Riemann--Stieltjes integral.
The conjugacy relation in \cref{eq:E-Lambda} follows from the reality of $X$.
Since the multipliers have modulus one and $X$ has bounded variation, $B^{\Lambda}(\gamma)$ is continuous and of bounded variation.
Put
\[
h_{\Lambda}=S(B^{\Lambda}(\gamma)).
\]
Define the orthogonal map
\begin{equation}\label{eq:T-Lambda}
T_{\Lambda}:E_{\Lambda}\to E_{\Lambda},
\qquad
(T_{\Lambda}z)_{\omega}=\e^{i\omega}z_{\omega}.
\end{equation}
The same notation denotes the induced graded automorphism of tensor series and signatures.

\subsection{Entire dependence of bounded-variation developments}

\begin{lemma}[Entire parameter dependence]\label{lem:parameter-entire}
Let $q,r\geq1$ and $T_0>0$, write $I=I_q$, let $\xi_1,\ldots,\xi_r:[0,T_0]\to\C$ be continuous bounded-variation paths, and let $K_1,\ldots,K_r:[0,T_0]\to M_q(\C)$ be continuous. Write
\[
\norm{K_j}_{\infty}=\sup_{0\leq s\leq T_0}\norm{K_j(s)}.
\]
For $u=(u_1,\ldots,u_r)\in\C^r$, let $V^u$ solve
\[
\dd V_s^u
 =V_s^u\sum_{j=1}^r u_j K_j(s)\dd\xi_j(s),
\qquad
V_0^u=I.
\]
Then $u\mapsto V_{T_0}^u$ is entire on $\C^r$.
Its Picard series converges absolutely and locally uniformly, and
\begin{equation}\label{eq:parameter-bound}
\norm{V_{T_0}^u}
 \leq
 \exp\!\left(
   \sum_{j=1}^r
   \abs{u_j}\,\norm{K_j}_{\infty}\,
   \Var(\xi_j;[0,T_0])
 \right).
\end{equation}
All mixed derivatives may be computed term by term from the Picard series.
\end{lemma}

\begin{proof}
The $n$th Picard term is the sum over words $j_1,\ldots,j_n\in\{1,\ldots,r\}$ of
\[
u_{j_1}\cdots u_{j_n}
\int_{0<s_1<\cdots<s_n<T_0}
K_{j_1}(s_1)\cdots K_{j_n}(s_n)
\dd\xi_{j_1}(s_1)\cdots\dd\xi_{j_n}(s_n).
\]
Let $\nu_j$ be the finite complex Stieltjes measure induced by the complex-valued path $\xi_j$, let $\abs{\nu_j}$ denote its total-variation measure, and define the finite positive measure
\[
\lambda_u
 =\sum_{j=1}^r
   \abs{u_j}\,\norm{K_j}_{\infty}\,\abs{\nu_j}.
\]
Continuity of $\xi_j$ implies that $\abs{\nu_j}$, and hence $\lambda_u$, is atomless.
After taking norms and summing over all words, the total contribution of degree $n$ is bounded by
\[
\lambda_u^{\otimes n}\!\left(
 \set{0<s_1<\cdots<s_n<T_0}
\right)
 =\frac{\lambda_u([0,T_0])^n}{n!}.
\]
The equality follows by partitioning $[0,T_0]^n$ into the $n!$ strict orderings; the diagonals are $\lambda_u^{\otimes n}$-null because $\lambda_u$ is atomless.
Summing over $n$ gives \eqref{eq:parameter-bound}.
The same estimate is locally uniform in $u$, so the Weierstrass test proves entireness and justifies termwise mixed differentiation of the Picard series.
\end{proof}

\subsection{The resonant identity}

Whenever $r\geq1$ is fixed below, $E_{ab}\in M_{r+1}(\C)$ denotes the matrix unit with a $1$ in row $a$ and column $b$, where the indices run from $0$ to $r$. Thus
\[
E_{ab}E_{cd}=\delta_{bc}E_{ad}.
\]
Here $\delta_{bc}$ is the Kronecker delta.

\begin{lemma}[Resonant skew-Hermitian endpoint identity]\label{lem:resonant-endpoint}
Let $\Lambda\subset 2\pi\Q$ be finite and symmetric, and choose $m\geq1$ such that
\[
m\omega\in2\pi\Z
\qquad(\omega\in\Lambda).
\]
Let $\Gamma=\gamma^{*m}$ be the $m$-fold concatenation of $\gamma$, naturally parametrised on $[0,mT]$.
Fix $r\geq1$, frequencies $\omega_1,\ldots,\omega_r\in\Lambda$, and real covectors $\eta_1,\ldots,\eta_r\in W^*$. No distinctness is assumed: the frequencies and covectors may repeat.
Choose real numbers $\sigma_0,\ldots,\sigma_r$ satisfying
\begin{equation}\label{eq:sigma-resonance}
\sigma_{j-1}-\sigma_j=\omega_j
\quad(1\leq j\leq r),
\qquad
m\sigma_j\in2\pi\Z
\quad(0\leq j\leq r).
\end{equation}
Set $D=i\diag(\sigma_0,\ldots,\sigma_r)$, write $I=I_{r+1}$, and, for $u=(u_1,\ldots,u_r)\in\R^r$, define
\begin{equation}\label{eq:rho-u}
\rho_u(a)
 =\ell(a)D
  +\sum_{j=1}^r
   u_j\eta_j(a)(E_{j-1,j}-E_{j,j-1}),
\end{equation}
first as a real-linear map $V\to\uLie(r+1)$. Denote its complex-linear extension by
\[
\rho_{u,\C}:V_{\C}\longrightarrow M_{r+1}(\C).
\]
If $U^u$ is the right development
\[
\dd U_s^u=U_s^u\rho_u(\dd\Gamma_s),
\qquad U_0^u=I,
\]
then
\begin{equation}\label{eq:eAu-identity}
U_{mT}^u=I
\qquad(u\in\R^r).
\end{equation}
\end{lemma}

\begin{proof}
The choice in \cref{eq:sigma-resonance} is always possible: take $\sigma_r=0$ and define recursively $\sigma_{j-1}=\sigma_j+\omega_j$. If $m\sigma_j\in2\pi\Z$, then
\[
m\sigma_{j-1}=m\sigma_j+m\omega_j\in2\pi\Z,
\]
so the resonance condition propagates backwards.

Since $g=\e^l$ in the complete graded algebra and all $m$ factors coincide,
\[
S(\Gamma)=g^m=(\e^l)^m=\e^{ml}.
\]
Here the formal exponential on the augmentation ideal and the formal logarithm on elements with constant term $1$ are mutually inverse. Hence
\[
\log S(\Gamma)=ml,
\]
so the logarithmic signature of $\Gamma$ is entire and its first level is $mv$.
The evaluated logarithm
\[
A(u)=\sum_{n\geq1}\rho_{u,\C}^{(n)}(m l_n)
\]
converges absolutely in operator norm, where $\rho_{u,\C}^{(n)}$ denotes the multiplicative extension to $V_{\C}^{\otimes n}$. Indeed, $l$ is entire and $\rho_{u,\C}$ is a bounded linear map on the finite-dimensional space $V_{\C}$. For real $u$, every partial sum is skew-Hermitian: each $m l_n$ is a real Lie polynomial and the matrices in \cref{eq:rho-u} are skew-Hermitian. Since $\uLie(r+1)$ is closed in operator norm, its limit satisfies $A(u)\in\uLie(r+1)$.

Apply \cref{thm:exact-isospectrality} to $\Gamma$ at $z=1$ with the complex-linear map $\rho_{u,\C}$. Since $\rho_{u,\C}(mv)=mD$,
\[
\det(\zeta I-A(u))=\det(\zeta I-mD).
\]
Equality of characteristic polynomials gives the same eigenvalue multiset, counted with algebraic multiplicity. By \cref{eq:sigma-resonance}, the eigenvalues of $mD$ lie in $2\pi i\Z$. Since $A(u)$ is skew-Hermitian, it is unitarily diagonalizable. Its eigenvalues therefore lie in $2\pi i\Z$, and hence $\exp A(u)=I$. Finally, \cref{lem:development-signature} identifies $U_{mT}^u$ with $\exp A(u)$, proving \cref{eq:eAu-identity}.
\end{proof}

\begin{lemma}[Continuous-BV gauge identity]\label{lem:resonant-gauge}
Under the hypotheses and notation of \cref{lem:resonant-endpoint}, write
\[
\Gamma_s=x_s v+X_s,
\qquad 0\leq s\leq mT,
\]
so that $x_0=0$ and $x_{mT}=m$, and set
\[
V_s^u=U_s^u\e^{-x_sD}.
\]
Then $V_0^u=V_{mT}^u=I$ and
\begin{equation}\label{eq:gauged-equation}
\dd V_s^u
 =V_s^u\sum_{j=1}^r u_j
 \left(
   \e^{i\omega_jx_s}E_{j-1,j}
   -\e^{-i\omega_jx_s}E_{j,j-1}
 \right)
 \eta_j(\dd X_s).
\end{equation}
\end{lemma}

\begin{proof}
For a continuous real bounded-variation path $x$, the matrix chain rule gives
\[
\dd(\e^{-x_sD})=-D\e^{-x_sD}\dd x_s.
\]
The continuous-BV product rule therefore yields
\begin{align*}
\dd V_s^u
&=(\dd U_s^u)\e^{-x_sD}+U_s^u\dd(\e^{-x_sD})\\
&=U_s^u\left(D\dd x_s+
 \sum_{j=1}^r u_j\eta_j(\dd X_s)(E_{j-1,j}-E_{j,j-1})\right)\e^{-x_sD}
 -U_s^uD\e^{-x_sD}\dd x_s.
\end{align*}
Because $D$ commutes with $\e^{-x_sD}$, the two terms containing $D\dd x_s$ cancel. Writing $U_s^u=V_s^u\e^{x_sD}$ leaves conjugation of the off-diagonal matrix units. For every $E_{ab}$,
\[
\e^{xD}E_{ab}\e^{-xD}
 =\e^{ix(\sigma_a-\sigma_b)}E_{ab}.
\]
Together with $\sigma_{j-1}-\sigma_j=\omega_j$, this gives exactly \cref{eq:gauged-equation}. The initial value is clear. At the endpoint, \cref{lem:resonant-endpoint} and $m\sigma_j\in2\pi\Z$ give
\[
V_{mT}^u=U_{mT}^u\e^{-mD}=I.
\]
\end{proof}

\begin{lemma}[Chain-coefficient extraction]\label{lem:chain-coefficient}
Under the hypotheses and notation of \cref{lem:resonant-endpoint},
\begin{equation}\label{eq:resonant-integral}
\int_{0<\tau_1<\cdots<\tau_r<mT}
\prod_{j=1}^r
\e^{i\omega_jx_{\tau_j}}
\eta_j(\dd X_{\tau_j})=0.
\end{equation}
\end{lemma}

\begin{proof}
For complex $u=(u_1,\ldots,u_r)\in\C^r$, let $V^u$ denote the unique solution of the complexified equation \cref{eq:gauged-equation}, with the same Stieltjes integrators and matrix coefficients. By \cref{lem:parameter-entire}, $u\mapsto V^u_{mT}$ is entire on $\C^r$, and its Picard series converges absolutely and locally uniformly. Put
\[
K_j(s)=
 \e^{i\omega_jx_s}E_{j-1,j}
 -\e^{-i\omega_jx_s}E_{j,j-1}.
\]
Then
\begin{equation}\label{eq:gauged-picard}
V_{mT}^u
 =I+\sum_{n\geq1}\ 
  \sum_{j_1,\ldots,j_n=1}^r
  u_{j_1}\cdots u_{j_n}
  \int_{0<\tau_1<\cdots<\tau_n<mT}
  K_{j_1}(\tau_1)\cdots K_{j_n}(\tau_n)
  \prod_{k=1}^n\eta_{j_k}(\dd X_{\tau_k}).
\end{equation}
The $n$th Picard term is homogeneous of total degree $n$ in $u$. For a power series in $u$, write $[u_1\cdots u_r]$ for the coefficient of the square-free monomial $u_1\cdots u_r$, and let $\mathfrak S_r$ denote the permutation group of $\{1,\ldots,r\}$. Hence only the degree-$r$ Picard term can contribute, and
\begin{align}
&[u_1\cdots u_r]\,(V_{mT}^u)_{0r} \notag\\
&\quad =\sum_{\pi\in\mathfrak{S}_r}
 \int_{0<\tau_1<\cdots<\tau_r<mT}
 \bigl(K_{\pi(1)}(\tau_1)\cdots K_{\pi(r)}(\tau_r)\bigr)_{0r}
 \prod_{k=1}^r\eta_{\pi(k)}(\dd X_{\tau_k}).
\label{eq:permutation-coefficient}
\end{align}
A nonzero $(0,r)$ entry in a product of $r$ adjacent off-diagonal matrix units corresponds to an index path starting at $0$, ending at $r$, and moving by $\pm1$ at every factor. Its total displacement is $r$, so every step must be upward. Hence the factors must be, in order,
\[
E_{0,1},E_{1,2},\ldots,E_{r-1,r},
\]
and the only surviving term in \cref{eq:permutation-coefficient} is $\pi=\id$, with the upper-triangular summand selected from each $K_j$. Since
\[
E_{0,1}E_{1,2}\cdots E_{r-1,r}=E_{0,r},
\]
the coefficient is precisely the integral in \cref{eq:resonant-integral}, with no sign or combinatorial factor.

By \cref{lem:resonant-gauge}, $V_{mT}^u=I$ for every real $u$. Since $u\mapsto V_{mT}^u$ is entire on $\C^r$, no identity theorem on the non-open set $\R^r$ is needed: fix $u_2,\ldots,u_r\in\R$ and apply the one-variable identity theorem in $u_1$, then repeat successively in $u_2,\ldots,u_r$. This proves $V_{mT}^u=I$ on all of $\C^r$. Hence every Taylor coefficient of $V_{mT}^u-I$ vanishes. For the square-free monomial, $\partial_{u_1}\cdots\partial_{u_r}$ at the origin equals its coefficient, so \cref{eq:resonant-integral} follows.
\end{proof}

\medskip
\noindent\textbf{The extraction when $r=2$.}
Writing
\[
K_1(s)=a_1(s)E_{0,1}-b_1(s)E_{1,0},
\qquad
K_2(s)=a_2(s)E_{1,2}-b_2(s)E_{2,1},
\]
with $a_j(s)=\e^{i\omega_jx_s}$ and $b_j(s)=\e^{-i\omega_jx_s}$, the coefficient of $u_1u_2$ in the $(0,2)$ entry of the degree-two Picard term is
\[
\int_{0<\tau_1<\tau_2<mT}
  \bigl(K_1(\tau_1)K_2(\tau_2)\bigr)_{0,2}
  \eta_1(\dd X_{\tau_1})\eta_2(\dd X_{\tau_2}).
\]
Indeed, $E_{0,1}E_{1,2}=E_{0,2}$, whereas the reversed order $K_2(\tau_1)K_1(\tau_2)$ has zero $(0,2)$ entry, and every term containing a downward matrix unit also has zero $(0,2)$ entry. Thus the displayed coefficient is exactly
\[
\int_{0<\tau_1<\tau_2<mT}
\e^{i\omega_1x_{\tau_1}}\e^{i\omega_2x_{\tau_2}}
\eta_1(\dd X_{\tau_1})\eta_2(\dd X_{\tau_2}),
\]
with neither a sign nor a factor of $2$.

\begin{lemma}[Coordinate separation for decorated spaces]\label{lem:decorated-coordinate-separation}
Let $\Lambda\subset 2\pi\Q$ be finite and symmetric. For $\eta\in W^*$, let $\eta_{\C}:W_{\C}\to\C$ denote its complex-linear extension. For $\omega\in\Lambda$, define
\[
\zeta_{\omega,\eta}(z)=\eta_{\C}(z_{\omega}),
\qquad z\in E_{\Lambda},
\]
and denote by the same symbol its unique complex-linear extension to $E_{\Lambda}\otimes_{\R}\C$. Then
\[
\operatorname{span}_{\C}
\set{\zeta_{\omega,\eta}:\omega\in\Lambda,\ \eta\in W^*}
 =\bigl(E_{\Lambda}\otimes_{\R}\C\bigr)^*.
\]
Consequently, for every $r\geq1$, the tensor products
\[
\zeta_{\omega_1,\eta_1}\otimes\cdots\otimes
\zeta_{\omega_r,\eta_r}
\]
with $\omega_j\in\Lambda$ and $\eta_j\in W^*$ separate $E_{\Lambda}^{\otimes r}$.
\end{lemma}

\begin{proof}
Choose $\Lambda_+\subset\Lambda\setminus\set{0}$ containing one element from each pair $\set{\omega,-\omega}$. Put
\[
E_0=
\begin{cases}
W, & 0\in\Lambda,\\
\set{0}, & 0\notin\Lambda.
\end{cases}
\]
The real-linear isomorphism
\[
\Psi:E_0\oplus\bigoplus_{\omega\in\Lambda_+}W_{\C}
   \longrightarrow E_{\Lambda}
\]
is given by
\[
\Psi\bigl(z_0,(z_{\omega})_{\omega\in\Lambda_+}\bigr)_0=z_0,
\qquad
\Psi(\cdots)_{\omega}=z_{\omega},
\qquad
\Psi(\cdots)_{-\omega}=\overline{z_{\omega}}
\quad(\omega\in\Lambda_+),
\]
with the zero-frequency coordinate omitted when $0\notin\Lambda$. Thus, as a real vector space,
\[
E_{\Lambda}\cong E_0\oplus\bigoplus_{\omega\in\Lambda_+}W_{\C}.
\]
Let $\eta^1,\ldots,\eta^{\dim W}$ be a real basis of $W^*$.

Work first on the underlying real vector space of a nonzero-frequency block. Write
\[
z_{\omega}=u+iw,
\qquad
z_{-\omega}=u-iw,
\qquad u,w\in W.
\]
Then
\[
\zeta_{\omega,\eta}=\eta(u)+i\eta(w),
\qquad
\zeta_{-\omega,\eta}=\eta(u)-i\eta(w),
\]
so
\[
\eta(u)=\frac{\zeta_{\omega,\eta}+\zeta_{-\omega,\eta}}{2},
\qquad
\eta(w)=\frac{\zeta_{\omega,\eta}-\zeta_{-\omega,\eta}}{2i}.
\]
Thus the functionals with frequencies $\omega$ and $-\omega$ recover the two real coordinate families $\eta^k(u)$ and $\eta^k(w)$, which form a real basis of the dual of the block. After complex-linear extension, they span the complexified dual of that block. At frequency zero, the $\zeta_{0,\eta^k}$ span the complexified dual of the $W$ block. The direct-sum decomposition proves the first assertion.

For the tensor-separation statement, use the canonical complex-linear identification
\[
\bigl(E_{\Lambda}^{\otimes_{\R} r}\bigr)\otimes_{\R}\C
\cong
\bigl(E_{\Lambda}\otimes_{\R}\C\bigr)^{\otimes_{\C} r}.
\]
Tensor products of the spanning family $\{\zeta_{\omega,\eta}\}$ span the dual of the complex tensor power on the right, and hence separate the complexification of $E_{\Lambda}^{\otimes_{\R} r}$. Since the canonical map
\[
E_{\Lambda}^{\otimes_{\R} r}
\longrightarrow
\bigl(E_{\Lambda}^{\otimes_{\R} r}\bigr)\otimes_{\R}\C
\]
is injective, the same tensor products separate the original real tensor space $E_{\Lambda}^{\otimes_{\R} r}$.
\end{proof}

\begin{proposition}[Rational-power resonance]\label{prop:resonance}
Let $\Lambda\subset 2\pi\Q$ be finite and symmetric.
Choose $m\geq 1$ such that
\[
m\omega\in 2\pi\Z
\qquad (\omega\in\Lambda).
\]
Then
\begin{equation}\label{eq:cyclic-identity}
h_{\Lambda}T_{\Lambda}(h_{\Lambda})\cdots
T_{\Lambda}^{m-1}(h_{\Lambda})=1.
\end{equation}
Equivalently, the decorated transverse path associated with the concatenation power $\gamma^{*m}$ has trivial signature.
\end{proposition}

\begin{proof}
Let $\Gamma=\gamma^{*m}$ and use its global decomposition $\Gamma_s=x_sv+X_s$. Fix an arbitrary tensor level $r\geq1$, frequencies $\omega_1,\ldots,\omega_r\in\Lambda$, and covectors $\eta_1,\ldots,\eta_r\in W^*$. Applying \cref{lem:resonant-endpoint,lem:resonant-gauge,lem:chain-coefficient} gives
\[
\int_{0<\tau_1<\cdots<\tau_r<mT}
\prod_{j=1}^r
\e^{i\omega_jx_{\tau_j}}
\eta_j(\dd X_{\tau_j})=0.
\]
By the definition of $B^{\Lambda}(\Gamma)$, the scalar identity above is exactly
\[
\left\langle
 \zeta_{\omega_1,\eta_1}\otimes\cdots\otimes
 \zeta_{\omega_r,\eta_r},
 S_r(B^{\Lambda}(\Gamma))
\right\rangle=0.
\]
The coordinate-separation statement in \cref{lem:decorated-coordinate-separation} shows that these decomposable functionals separate $E_{\Lambda}^{\otimes r}$. Since the frequencies and covectors were arbitrary,
\[
S_r(B^{\Lambda}(\Gamma))=0.
\]
As $r\geq1$ was arbitrary, every positive tensor level vanishes. Hence
\[
S(B^{\Lambda}(\Gamma))=1.
\]

For $0\leq j\leq m-1$ and $0\leq s\leq T$, the global decomposition on the $j$th copy satisfies
\[
x_{jT+s}=j+x_s,
\qquad
X_{jT+s}=X_s,
\]
where the quantities on the right refer to the original copy of $\gamma$. Hence, in the $\omega$ coordinate,
\[
\e^{i\omega x_{jT+s}}\dd X_{jT+s}
 =\e^{ij\omega}\e^{i\omega x_s}\dd X_s.
\]
After translating the decorated segment to start at zero, it is therefore $T_{\Lambda}^jB^{\Lambda}(\gamma)$. Functoriality of signatures and Chen multiplication give the path-level concatenation identity
\[
S(B^{\Lambda}(\gamma^{*m}))
 =\prod_{j=0}^{m-1}S(T_{\Lambda}^jB^{\Lambda}(\gamma))
 =h_{\Lambda}T_{\Lambda}(h_{\Lambda})\cdots
  T_{\Lambda}^{m-1}(h_{\Lambda}).
\]
The left-hand side is $1$ by the preceding argument, which proves \cref{eq:cyclic-identity}.
\end{proof}

\paragraph{Integral-frequency specialisation and comparison with BGW.}
When $\Lambda\subset2\pi\Z$, one may take $m=1$, so \cref{eq:cyclic-identity} reduces to
\[
S(B^{\Lambda}(\gamma))=h_{\Lambda}=1.
\]
In the planar bounded-variation setting, the nondegenerate logarithmic-signature identities of \cite[Theorem~6.5]{BGW2025} follow from this full decorated-signature identity.

\section{Signature-tree geometry: reduction and factorisation}\label{sec:tree-factorisation}

The complete Le Donne--Z\"ust signature tree is used first to turn the cyclic identity of \cref{prop:resonance} into a fixed point on an endpoint geodesic, and later to turn a factorisation $g=a\e^v a^{-1}$ into the reduced path decomposition. For points $x,y$ in a metric tree, $[x,y]$ denotes their unique geodesic segment; for a nonempty subset $C$, $d(x,C):=\inf_{c\in C}d(x,c)$. For tree-reduced paths, \cref{prop:reducedness,prop:prefix-geodesic} identify the endpoint geodesic with the prefix-signature trace. The standard metric-tree facts required below are collected in \cref{app:metric-tree}.

\subsection{The inverse-limit metric}

Let $E$ be a finite-dimensional real Hilbert space. Define the rectifiable signature group by
\[
\Sl(E)=\set{S(\eta):\eta\text{ is a rectifiable }E\text{-valued path}}.
\]
For each $N\geq1$, let $G_N(E)$ be the free step-$N$ nilpotent Lie group over $E$, equipped with its left-invariant Carnot--Carath\'eodory metric $d_N$ induced by the Hilbert norm on the horizontal layer, and let $1_N$ denote its identity.
Write
\[
\pi_N^{N+1}:G_{N+1}(E)\longrightarrow G_N(E)
\]
for the canonical truncation homomorphism.
Following Le Donne--Z\"ust \cite[Section~2]{LeDonneZust2021}, define the \emph{finite-distance inverse limit}
\[
\begin{aligned}
G_{\infty}(E)
=\Bigl\{(g_N)_{N\geq1}\in\prod_{N\geq1}G_N(E):\;&
\pi_N^{N+1}(g_{N+1})=g_N\quad(N\geq1),\\
&\sup_{N\geq1}d_N(1_N,g_N)<\infty
\Bigr\}.
\end{aligned}
\]
For each $N$, let
\[
\pi_N:G_{\infty}(E)\longrightarrow G_N(E),
\qquad
\pi_N((g_k)_{k\geq1})=g_N,
\]
denote the coordinate projection.
For $g=(g_N)$ and $h=(h_N)$ in $G_{\infty}(E)$, set
\begin{equation}\label{eq:d-infty}
 d_{\infty}(g,h)
 =\sup_{N\geq1}d_N(g_N,h_N).
\end{equation}
This is finite, since
\[
d_N(g_N,h_N)
\leq d_N(g_N,1_N)+d_N(1_N,h_N)
\]
and both suprema on the right are finite by definition. Thus $G_{\infty}(E)$ is not the unrestricted algebraic inverse limit of all compatible truncations; it is the finite-$d_{\infty}$-distance class of the identity.

If $g=S(\eta)$ is a rectifiable signature and $\pi_Ng$ denotes its step-$N$ truncation, then the canonical horizontal lift of $\eta$ to $G_N(E)$ has length $L(\eta)$. Hence
\[
d_N(1_N,\pi_Ng)\leq L(\eta)
\qquad(N\geq1),
\]
so $(\pi_Ng)_{N\geq1}$ belongs to $G_{\infty}(E)$.

\begin{theorem}[Le Donne--Z\"ust]\label{thm:LDZ}
The space $G_{\infty}(E)$ is a group under coordinatewise multiplication and inversion, and $d_{\infty}$ is a finite left-invariant complete metric on it \cite[Lemmas~2.2(2) and~3.1]{LeDonneZust2021}. The truncation map
\[
\iota:\Sl(E)\longrightarrow G_{\infty}(E),
\qquad
\iota(g)=(\pi_Ng)_{N\geq1},
\]
is a group isomorphism \cite[Theorem~4.4]{LeDonneZust2021}. We use $\iota$ to identify $\Sl(E)$ with $G_{\infty}(E)$ and transport $d_{\infty}$ to $\Sl(E)$. Under this identification, $(\Sl(E),d_{\infty})$ is a complete metric tree \cite[Theorem~4.4]{LeDonneZust2021}. In particular, every point of this complete space is the signature of a rectifiable $E$-valued path; no additional non-rectifiable completion points are introduced.
\end{theorem}

Under the identification in \cref{thm:LDZ}, we write $\pi_Ng=g_N$ for the $N$th truncation of $g\in\Sl(E)$.
For $x,y\in\Sl(E)$, the notation $[x,y]$ denotes the unique geodesic segment joining them in the signature tree.

Only left translations are used as metric isometries below; no continuity or isometric property of right translations is assumed.

The next two functorial properties follow directly from the definition.

\begin{lemma}[Orthogonal functoriality and prefix continuity]\label{lem:prefix-continuity}
Let $U:E\to E$ be orthogonal.
\begin{enumerate}[label=\textup{(\alph*)}]
\item The induced compatible family of graded automorphisms preserves $G_{\infty}(E)$; under the identification in \cref{thm:LDZ}, the corresponding signature automorphism $U_*$ is an isometry of $(\Sl(E),d_{\infty})$.
\item If $\beta:[0,T]\to E$ is rectifiable and
\[
r_t=S(\beta|_{[0,t]}),
\]
then $t\mapsto r_t$ is continuous in $d_{\infty}$, and for $s<t$,
\begin{equation}\label{eq:prefix-bound}
d_{\infty}(r_s,r_t)
 \leq L(\beta|_{[s,t]}).
\end{equation}
\end{enumerate}
\end{lemma}

\begin{proof}
For each $N$, the graded automorphism $U_{*,N}$ induced by $U$ sends horizontal curves in $G_N(E)$ to horizontal curves of the same length, commutes with truncation, and fixes $1_N$. It is therefore an isometry for $d_N$, and
\[
d_N\bigl(1_N,U_{*,N}(g_N)\bigr)=d_N(1_N,g_N).
\]
Hence the induced compatible family preserves the defining finite-distance condition for $G_{\infty}(E)$. Taking the supremum over $N$ proves part~\textup{(a)}.

For part (b), Chen's identity and left invariance give
\[
d_{\infty}(r_s,r_t)
 =d_{\infty}\!\left(1,S(\beta|_{[s,t]})\right).
\]
At every truncation level, the canonical horizontal lift of $\beta|_{[s,t]}$ has length $L(\beta|_{[s,t]})$, so its endpoint distance is no larger.
Taking the supremum over $N$ proves \cref{eq:prefix-bound}.
Continuity follows from continuity of the variation function of a continuous bounded-variation path.
\end{proof}

\begin{proposition}[Reducedness and the prefix-signature lift]\label{prop:reducedness}
Let $E$ be a finite-dimensional real normed vector space and let $\eta:[0,T]\to E$ be a continuous bounded-variation path. Put
\[
q(t)=\Var(\eta;[0,t]),
\qquad L=q(T),
\]
and let $\bar\eta:[0,L]\to E$ be the variation parametrisation, so that
\[
\eta_t=\bar\eta_{q(t)},
\qquad
\Var(\bar\eta;[u_1,u_2])=u_2-u_1.
\]
When $L=0$, $\bar\eta$ is understood to be constant. Set
\[
r_t=S(\eta|_{[0,t]}),
\qquad 0\leq t\leq T.
\]
The following are equivalent:
\begin{enumerate}[label=\textup{(\alph*)}]
\item $\bar\eta$ is the Hambly--Lyons reduced representative of the tree-like equivalence class of $\eta$;
\item every restriction $\eta|_{[s,t]}$ with trivial signature is constant;
\item if $r_s=r_t$ for $0\leq s<t\leq T$, then $\eta|_{[s,t]}$ is constant; equivalently, the prefix-signature lift is injective after constant waiting intervals are collapsed.
\end{enumerate}
In particular, every restriction of a path satisfying these conditions again satisfies them.
\end{proposition}

\begin{proof}
Chen's identity gives
\[
r_s^{-1}r_t=S(\eta|_{[s,t]}),
\]
so \textup{(b)} and \textup{(c)} are equivalent.

Assume \textup{(a)}, and suppose that $S(\eta|_{[s,t]})=1$. Its first level then vanishes, so $\eta_s=\eta_t$. If the subpath were nonconstant, deleting the interval $[s,t]$ would produce a continuous bounded-variation path $\eta^{\sharp}$ with
\[
S(\eta^{\sharp})
 =S(\eta|_{[0,s]})S(\eta|_{[t,T]})
 =S(\eta)
\]
by Chen's identity, whereas additivity of variation would give
\[
\Var(\eta^{\sharp})
 =\Var(\eta)-\Var(\eta;[s,t])
 <\Var(\eta)=L.
\]
This contradicts the minimal-length property of the Hambly--Lyons reduced representative $\bar\eta$. Hence \textup{(b)} holds.

Conversely, assume \textup{(c)}. If $L=0$, the claim is immediate. For
\[
\bar r_u=S(\bar\eta|_{[0,u]}),
\qquad 0\leq u\leq L,
\]
we have $r_t=\bar r_{q(t)}$ by reparametrisation invariance of the signature. If $0\leq u_1<u_2\leq L$ and $\bar r_{u_1}=\bar r_{u_2}$, choose $t_1\leq t_2$ with $q(t_j)=u_j$. Then $r_{t_1}=r_{t_2}$, so \textup{(c)} makes $\eta|_{[t_1,t_2]}$ constant. This is impossible because
\[
\Var(\eta;[t_1,t_2])=u_2-u_1>0.
\]
Thus $u\mapsto\bar r_u$ is injective.

We now verify the analytic hypotheses of the reduced-signature-path theorem. Following \cite[Definition~2.1]{BGLY2016}, let $G^{(*)}_{\mathrm{p.r.c.}}(E)$ denote the group-like series
\[
a=1+\sum_{n\geq1}a_n
\qquad\text{such that}\qquad
\sup_{n\geq1}\norm{a_n}_{\pi,n}^{1/n}<\infty,
\]
and equip this space with the metric
\[
d_{\mathrm{BGLY}}(a,b)
 =\sup_{n\geq1}
   \norm{(a^{-1}b)_n}_{\pi,n}^{1/n},
\]
where $(a^{-1}b)_n$ denotes the homogeneous degree-$n$ component; this avoids conflict with the step-truncation notation $\pi_N$ used above. The projective tensor norms used here satisfy the admissibility hypotheses of \cite[Section~2.1]{BGLY2016}. The factorial estimate gives, for every $u\in[0,L]$,
\[
\sup_{n\geq1}\norm{S_n(\bar\eta|_{[0,u]})}_{\pi,n}^{1/n}
 \leq \sup_{n\geq1}\frac{u}{(n!)^{1/n}}
 \leq u,
\]
so $\bar r_u\in G^{(*)}_{\mathrm{p.r.c.}}(E)$. Moreover, for $0\leq u_1<u_2\leq L$, Chen's identity and the unit-speed variation parametrisation give
\[
\begin{aligned}
d_{\mathrm{BGLY}}(\bar r_{u_1},\bar r_{u_2})
&=\sup_{n\geq1}
  \norm{S_n(\bar\eta|_{[u_1,u_2]})}_{\pi,n}^{1/n}\\
&\leq \sup_{n\geq1}
  \frac{u_2-u_1}{(n!)^{1/n}}
 \leq u_2-u_1.
\end{aligned}
\]
Thus $\bar r$ is $1$-Lipschitz in that metric, hence continuous and of finite $1$-variation.

Apply the reduced-signature-path theorem of Boedihardjo--Geng--Lyons--Yang to $\bar r$ \cite[Lemma~4.6]{BGLY2016}. It produces an injective path $\widetilde r$ with the same endpoints, unique up to reparametrisation. Since $\bar r$ is already injective, the uniqueness clause makes $\bar r$ a reparametrisation of $\widetilde r$. In the finite-$1$-variation case used here, the first-level projection of $\widetilde r$ is the Hambly--Lyons reduced representative \cite[Remark~4.1]{BGLY2016}. Because
\[
\pi_1(\bar r_u)=\bar\eta_u-\bar\eta_0,
\]
it follows that $\bar\eta$ itself represents that reduced path. Hence \textup{(a)} holds.

Finally, condition \textup{(b)} is inherited by restrictions, proving the last assertion.
\end{proof}

\begin{proposition}[Weak reparametrisation and reduced uniqueness]\label{prop:weak-reparametrisation}
Let $E$ be a finite-dimensional real vector space, and let $\xi:[0,A]\to E$ and $\widetilde\xi:[0,B]\to E$ be continuous bounded-variation paths with $\xi\sim_{\mathrm{rep}}\widetilde\xi$, witnessed by continuous nondecreasing surjections $\sigma:[0,1]\to[0,A]$ and $\widetilde\sigma:[0,1]\to[0,B]$.
\begin{enumerate}[label=\textup{(\alph*)}]
\item If $0\leq u\leq w\leq1$, $s=\sigma(u)$, $t=\sigma(w)$, $\widetilde s=\widetilde\sigma(u)$, and $\widetilde t=\widetilde\sigma(w)$, then
\[
\xi|_{[s,t]}\sim_{\mathrm{rep}}
\widetilde\xi|_{[\widetilde s,\widetilde t]}.
\]
The two subpaths have equal signatures and are constant simultaneously.
\item The equivalent reducedness conditions in \cref{prop:reducedness} are invariant under $\sim_{\mathrm{rep}}$.
\item If neither path has a nontrivial constant subinterval, then there is an increasing homeomorphism $h:[0,A]\to[0,B]$ such that
\[
\xi=\widetilde\xi\circ h.
\]
\end{enumerate}
Consequently, two tree-reduced bounded-variation paths with the same initial point and the same signature are weakly reparametrisation-equivalent.
\end{proposition}

\begin{proof}
Write
\[
\zeta=\xi\circ\sigma=\widetilde\xi\circ\widetilde\sigma.
\]
For $0\leq u\leq w\leq1$, the restrictions of $\sigma$ and $\widetilde\sigma$ to $[u,w]$ are nondecreasing surjections onto $[s,t]$ and $[\widetilde s,\widetilde t]$, respectively. They therefore witness the asserted weak reparametrisation of the two subpaths. Signature invariance under nondecreasing surjective reparametrisation gives equality of their signatures. Surjectivity also shows that one subpath is constant if and only if $\zeta|_{[u,w]}$ is constant, which is equivalent to constancy of the other subpath. This proves part \textup{(a)}.

Every closed subinterval of either original path is obtained from a suitable interval $[u,w]$. Part \textup{(a)} therefore transfers the absence of nonconstant trivial-signature subpaths from one path to the other. By \cref{prop:reducedness}, this proves part \textup{(b)}.

Assume now that neither path has a nontrivial constant subinterval. If $\sigma(u)=\sigma(w)$, then $\zeta$ is constant on $[u,w]$; hence $\widetilde\sigma(u)=\widetilde\sigma(w)$, since otherwise $\widetilde\xi$ would be constant on a nontrivial interval. The converse follows symmetrically, so $\sigma$ and $\widetilde\sigma$ have the same fibres. Therefore
\[
h(s)=\widetilde\sigma(u)
\qquad\text{for any }u\in\sigma^{-1}(s)
\]
is well defined. It is an order isomorphism from $[0,A]$ onto $[0,B]$, hence an increasing homeomorphism, and
\[
\widetilde\xi(h(s))
 =\widetilde\xi(\widetilde\sigma(u))
 =\zeta(u)
 =\xi(\sigma(u))
 =\xi(s).
\]
This proves part \textup{(c)}.

For the final assertion, suppress waiting intervals by variation parametrisation. By \cref{prop:reducedness}, the resulting paths are Hambly--Lyons reduced representatives. If the original paths have the same initial point and signature, Hambly--Lyons uniqueness \cite[Corollary~1.6]{HamblyLyons2010} identifies the two waiting-free representatives up to increasing reparametrisation. Restoring the suppressed waiting intervals gives weak reparametrisation of the original paths.
\end{proof}

\begin{proposition}[Prefix trace of a tree-reduced path]\label{prop:prefix-geodesic}
Let $E$ be a finite-dimensional real vector space equipped with an auxiliary Euclidean norm, and let $\eta:[0,T]\to E$ satisfy the equivalent conditions of \cref{prop:reducedness}. Then
\begin{equation}\label{eq:prefix-trace-geodesic}
\set{S(\eta|_{[0,t]}):0\leq t\leq T}
 =[1,S(\eta)]
\end{equation}
in the Le Donne--Z\"ust signature tree. Every point of the endpoint geodesic is therefore attained by a prefix of the original parametrised path. The geodesic trace and its order are independent of the auxiliary Euclidean norm.
\end{proposition}

\begin{proof}
Write $r_t=S(\eta|_{[0,t]})$. By \cref{lem:prefix-continuity}, $r$ is continuous. Let
\[
s(t)=\Var(\eta;[0,t]),
\qquad L=s(T).
\]
If $L=0$, the claim is immediate. Assume $L>0$, and let $\bar\eta:[0,L]\to E$ be the variation parametrisation, so that
\[
\eta_t=\bar\eta_{s(t)},
\qquad
\Var(\bar\eta;[u_1,u_2])=u_2-u_1.
\]
The map $s:[0,T]\to[0,L]$ is continuous, nondecreasing, and onto. With
\[
\bar r_u=S(\bar\eta|_{[0,u]}),
\]
reparametrisation invariance gives $r_t=\bar r_{s(t)}$ and hence $r([0,T])=\bar r([0,L])$.

The map $\bar r$ is injective. If $0\leq u_1<u_2\leq L$ and $\bar r_{u_1}=\bar r_{u_2}$, choose $t_1\leq t_2$ with $s(t_j)=u_j$. Then
\[
S(\eta|_{[t_1,t_2]})=1.
\]
By \cref{prop:reducedness}, the subpath is constant, whereas
\[
\Var(\eta;[t_1,t_2])=u_2-u_1>0,
\]
a contradiction. Thus $\bar r$ is a continuous injection from a compact interval into the Hausdorff space $\Sl(E)$, hence a homeomorphism onto an arc from $1$ to $S(\eta)$. By \cref{thm:LDZ}, this is the unique geodesic $[1,S(\eta)]$. Surjectivity of $s$ shows that every point is attained at an original prefix time.

If two Euclidean norms satisfy $c\norm{x}_1\leq\norm{x}_2\leq C\norm{x}_1$, their truncation-level Carnot--Carath\'eodory metrics, and hence their $d_\infty$ metrics, satisfy the same bi-Lipschitz comparison. They therefore define the same continuous arcs and the same order on each arc. This proves norm-independence.
\end{proof}

\subsection{Metric-tree inputs}

The fixed-point, projection, gate, and translation-axis facts used below are stated in \cref{app:metric-tree}. They are standard consequences of complete real-tree geometry and contain no signature-specific algebra.

\subsection{Straight signature lines and path reduction}

We now apply these metric-tree facts to rectifiable signature lines and then return from group factorisation to path reduction.

\begin{lemma}[Straight signature line]\label{lem:straight-signature-line}
Let $E$ be a finite-dimensional real Hilbert space, denote its norm by $\norm{\cdot}_E$, and let $0\neq v\in E$.
Then
\[
\mathcal{A}_v=\set{\e^{tv}:t\in\R}
\]
is a closed geodesic line in $(\Sl(E),d_\infty)$, and
\begin{equation}\label{eq:line-distance}
d_\infty(\e^{sv},\e^{tv})=\abs{t-s}\,\norm{v}_E
\qquad (s,t\in\R).
\end{equation}
If $a\in\Sl(E)$ and $g=a\e^v a^{-1}$, then
\[
\mathcal{A}=a\mathcal{A}_v
 =\set{a\e^{tv}:t\in\R}
\]
is a closed convex geodesic line, invariant under left multiplication by $g$, and
\begin{equation}\label{eq:axis-translation}
g\,a\e^{tv}=a\e^{(t+1)v}.
\end{equation}
\end{lemma}

\begin{proof}
By left invariance,
\[
d_\infty(\e^{sv},\e^{tv})
 =d_\infty(1,\e^{(t-s)v}).
\]
At every truncation level, the straight horizontal path $u\mapsto\e^{u(t-s)v}$ has length $\abs{t-s}\norm{v}_E$, so the distance is at most this number.
At level one, the Carnot--Carath\'eodory metric is the Euclidean norm, giving the reverse inequality.
This proves \cref{eq:line-distance}.
Hence $t\mapsto\e^{tv}$ is, up to the constant speed $\norm{v}_E$, an isometric embedding of $\R$; its image is a closed geodesic line and therefore convex.
Left translation by $a$ preserves all of these properties.
Finally,
\[
g\,a\e^{tv}
 =a\e^v a^{-1}a\e^{tv}
 =a\e^{(t+1)v},
\]
which proves invariance and \cref{eq:axis-translation}.
\end{proof}

\begin{proposition}[Geometric axis and gate decomposition]\label{prop:axis-gate-decomposition}
Let $V$ be a finite-dimensional real normed space, let $0\neq v\in V$, and let $g,a\in\Sl(V)$ satisfy
\[
g=a\e^v a^{-1}.
\]
Equip $V$ with an auxiliary Euclidean norm, denoted by $\norm{\cdot}_{\mathrm{Euc}}$, and set
\[
\mathcal{A}=a\mathcal{A}_v
 =\set{a\e^{tv}:t\in\R},
\qquad L_g(x)=gx,
\]
and let $p$ be the metric projection of $1$ onto $\mathcal{A}$. Then, for every $x\in\Sl(V)$,
\begin{align}
 d_{\infty}(x,gx)
 &=2d_{\infty}(x,\mathcal{A})+\norm{v}_{\mathrm{Euc}},
 \label{eq:displacement-axis}\\
 \mathcal{A}
 &=\set{x\in\Sl(V):d_{\infty}(x,gx)=\norm{v}_{\mathrm{Euc}}}.
 \label{eq:min-displacement-axis}
\end{align}
Thus $\mathcal{A}$ is the unique minimum-displacement axis of $L_g$ and is independent of the chosen factorisation of $g$. Moreover,
\begin{equation}\label{eq:conjugacy-geodesic-decomposition}
[1,g]=[1,p]\cup[p,gp]\cup[gp,g],
\end{equation}
where the three pieces meet only at $p$ and $gp$, and
\begin{equation}\label{eq:reduced-conjugator-identity}
gp=p\e^v,
\qquad
g=p\e^v p^{-1}.
\end{equation}
Finally,
\begin{equation}\label{eq:axis-intersection}
\mathcal{A}\cap[1,g]=[p,gp].
\end{equation}
The first and third pieces in \cref{eq:conjugacy-geodesic-decomposition} degenerate simultaneously.
\end{proposition}

\begin{proof}
By \cref{lem:straight-signature-line}, $\mathcal{A}$ is a closed geodesic line invariant under $L_g$, and
\[
\alpha_{\mathcal A}(t)=a\e^{(t/\norm{v}_{\mathrm{Euc}})v}
\]
is an isometric parametrisation satisfying
$L_g(\alpha_{\mathcal A}(t))=\alpha_{\mathcal A}(t+\norm{v}_{\mathrm{Euc}})$.
Applying \cref{prop:metric-tree-axis} gives \cref{eq:displacement-axis,eq:min-displacement-axis} and shows that $gp$ is the projection of $g$ onto $\mathcal{A}$. The same proposition gives \cref{eq:conjugacy-geodesic-decomposition} and its intersection properties.

Since $p\in\mathcal{A}$, write $p=a\e^{sv}$. Then
\[
gp=a\e^{(s+1)v}=p\e^v,
\]
which gives \cref{eq:reduced-conjugator-identity}. Convexity yields $[p,gp]\subset\mathcal{A}$. The projection property excludes points of $[1,p)$ and $(gp,g]$ from $\mathcal{A}$, proving \cref{eq:axis-intersection}. The final assertion follows from $gp=g$ if and only if $p=1$.
\end{proof}

The resulting axis--gate decomposition is illustrated in
\cref{fig:axis-gate}.

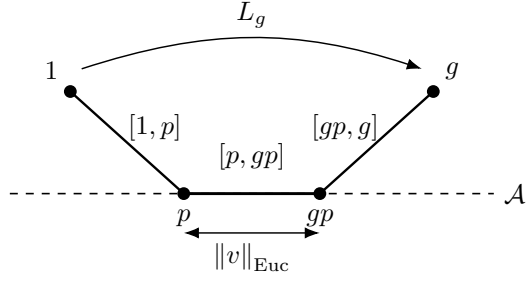
\begin{figure}[ht]
\centering
\begin{tikzpicture}[
  every node/.style={font=\small},
  point/.style={circle, fill=black, inner sep=1.7pt},
  move/.style={-{Latex[length=2mm]}, line width=0.55pt}
]
\coordinate (one) at (-2.4,1.35);
\coordinate (p) at (-0.9,0);
\coordinate (gp) at (0.9,0);
\coordinate (g) at (2.4,1.35);

\draw[dashed, line width=0.6pt] (-3.2,0)--(3.2,0);
\node[right] at (3.2,0) {$\mathcal{A}$};
\draw[line width=1.1pt] (p)--(gp);
\draw[line width=0.9pt] (one)--(p);
\draw[line width=0.9pt] (gp)--(g);

\node[point] at (one) {};
\node[above left=1pt of one] {$1$};
\node[point] at (p) {};
\node[below=2pt of p] {$p$};
\node[point] at (gp) {};
\node[below=2pt of gp] {$gp$};
\node[point] at (g) {};
\node[above right=1pt of g] {$g$};

\node[above left=16pt and -3pt of p] {$[1,p]$};
\node[above=5pt of p, xshift=0.9cm] {$[p,gp]$};
\node[above right=16pt and -7pt of gp] {$[gp,g]$};

\draw[move] (-2.25,1.65) to[bend left=18]
  node[midway, above] {$L_g$} (2.25,1.65);
\draw[{Latex[length=2mm]}-{Latex[length=2mm]}, line width=0.55pt]
  (-0.9,-0.52)--(0.9,-0.52)
  node[midway, below] {$\norm{v}_{\mathrm{Euc}}$};
\end{tikzpicture}
\caption{Axis and gate decomposition in the signature tree. Here $gp=p\e^v$, the geodesic $[1,g]$ meets the invariant axis exactly in $[p,gp]$, and $d_{\infty}(p,gp)=\norm{v}_{\mathrm{Euc}}$.}
\label{fig:axis-gate}
\end{figure}

\begin{theorem}[Reduction lemma: reduced-group and path conjugacy]\label{thm:reduction-lemma}
Let $V$ be a finite-dimensional real normed space and let $\gamma:[0,T]\to V$ be a continuous tree-reduced bounded-variation path, translated so that $\gamma_0=0$. Set
\[
g=S(\gamma),\qquad v=S_1(\gamma)=\gamma_T\neq0.
\]
The following are equivalent.
\begin{enumerate}[label=\textup{(\alph*)}]
\item There is $a\in\Sl(V)$ such that $g=a\e^v a^{-1}$.
\item There is $t_0\in[0,T]$ such that, for $\beta=\gamma|_{[0,t_0]}$,
\[
\gamma\sim_{\mathrm{rep}}
\beta*\lambda_v*\overleftarrow{\beta}.
\]
\end{enumerate}
Under \textup{(a)}, let $p$ be the gate from $1$ to the axis in \cref{prop:axis-gate-decomposition}. One may choose $S(\beta)=p$, and this prefix signature is independent of the auxiliary Euclidean norm.
\end{theorem}

\begin{proof}
Condition \textup{(b)} implies \textup{(a)} by Chen's identity. Assume \textup{(a)} and use the notation of \cref{prop:axis-gate-decomposition}.
For
\[
r_t=S(\gamma|_{[0,t]}),
\]
\cref{prop:prefix-geodesic} gives $r([0,T])=[1,g]$. Hence $p$ is attained at some $t_0\in[0,T]$. Set $\beta=\gamma|_{[0,t_0]}$. By \cref{prop:reducedness}, $\beta$ is tree-reduced, and \cref{prop:prefix-geodesic} identifies its prefix trace with $[1,p]$.

Let
\[
\delta=\beta*\lambda_v*\overleftarrow{\beta}.
\]
After waiting intervals are suppressed, the first part of its prefix-signature lift traverses $[1,p]$ injectively. The line segment traverses
\[
\set{p\e^{uv}:0\leq u\leq1}=[p,gp]
\]
injectively by \cref{eq:line-distance}. To describe the reversed final part, let $q=S(\beta|_{[0,s]})$ be a prefix signature of $\beta$. The corresponding partial reverse has signature $p^{-1}q$, so the full prefix signature of $\delta$ at that point is
\[
(p\e^v)(p^{-1}q)=gq.
\]
Thus this part traverses the left translate $g[p,1]=[gp,g]$ injectively. The intersection properties in \cref{eq:conjugacy-geodesic-decomposition} show that the three traces meet only at their consecutive endpoints. Hence the full prefix-signature lift of $\delta$ is injective, and \cref{prop:reducedness} shows that $\delta$ is tree-reduced.

By \cref{eq:reduced-conjugator-identity},
\[
S(\delta)=p\e^v p^{-1}=g.
\]
Hence $\delta$ and $\gamma$ are tree-reduced representatives of the same tree-like equivalence class. The final assertion of \cref{prop:weak-reparametrisation} gives
\[
\gamma\sim_{\mathrm{rep}}\delta
 =\beta*\lambda_v*\overleftarrow{\beta}.
\]

For norm-independence, the algebraic axis $\mathcal{A}=a\mathcal{A}_v$ is fixed, while \cref{eq:axis-intersection} characterises $p$ as the first point of $\mathcal{A}\cap[1,g]$ encountered from $1$. By \cref{prop:prefix-geodesic}, equivalent auxiliary Euclidean norms give the same arc $[1,g]$ with the same order. They therefore give the same gate $p$.
\end{proof}

\begin{corollary}[Unique reduced conjugating prefix]\label{cor:canonical-prefix}
Under the hypotheses of \cref{thm:reduction-lemma}, assume its equivalent conditions. The minimum-displacement axis $\mathcal{A}$ and its gate $p$ from $1$ are determined by $g$. The based prefix $\beta$ with $S(\beta)=p$ is the unique prefix that yields a reduced literal path conjugation, up to waiting intervals and weak increasing reparametrisation.
\end{corollary}

\begin{proof}
Suppose that $\beta'$ is a based prefix of $\gamma$, put $b=S(\beta')$, and assume
\[
\gamma\sim_{\mathrm{rep}}
\beta'*\lambda_v*\overleftarrow{\beta'}.
\]
Then $g=b\e^v b^{-1}$, so \cref{eq:min-displacement-axis} gives $b\mathcal{A}_v=\mathcal{A}$. Since $b\in[1,g]$, \cref{eq:axis-intersection} gives
\[
b\in[p,gp].
\]
Write $b=p\e^{sv}$ with $s\in[0,1]$. If $s>0$, the middle part of the prefix trace of
$\beta'*\lambda_v*\overleftarrow{\beta'}$ contains the nondegenerate segment $[gp,gb]$, while the final part contains the same segment as
\[
g[p,b]=[gp,gb],
\]
traversed in the opposite direction. Because this path is weakly reparametrisation-equivalent to the tree-reduced path $\gamma$, part~\textup{(b)} of \cref{prop:weak-reparametrisation} makes it tree-reduced. Choose an interior point $q\in(gp,gb)$. During the middle line segment, the prefix-signature lift attains $q$ before reaching $gb$; during the final reversed copy, it attains the same point $q$ again after leaving $gb$. These are distinct prefix times and the intervening trace is nonconstant. This contradicts the injectivity of the waiting-free prefix-signature lift in \cref{prop:reducedness}. Hence $s=0$ and $b=p$.

If two prefix times give the signature $p$, the intervening subpath has trivial signature and is therefore constant by tree-reducedness. The corresponding prefixes differ only by a terminal waiting interval, proving uniqueness under $\sim_{\mathrm{rep}}$.
\end{proof}

\begin{corollary}[Equivalent modified formulations]\label{cor:formulations-equivalent}
Let $V$ be a finite-dimensional real normed space, let $\gamma:[0,T]\to V$ be a continuous tree-reduced bounded-variation path, and put $v=\gamma_T-\gamma_0$. If $v\neq0$, then the following formulations are equivalent:
\begin{enumerate}[label=\textup{(\roman*)}]
\item centred weak path conjugacy to a line in the sense of \cref{def:two-conjugacies};
\item conjugacy to $\e^v$ in the reduced bounded-variation path group;
\item a signature identity $S(\gamma)=a\e^v a^{-1}$ with $a\in\Sl(V)$.
\end{enumerate}
When $v=0$, the corresponding formulations all reduce to the constant reduced path.
\end{corollary}

\begin{remark}[Factorising prefixes and the gate]\label{rem:two-prefixes}
Under the hypotheses of \cref{thm:reduction-lemma}, let $p$ be the gate from \cref{cor:canonical-prefix}. Every $q\in[p,gp]$ lies on the invariant line and satisfies
\[
g=q\e^v q^{-1}.
\]
By \cref{eq:conjugacy-geodesic-decomposition,prop:prefix-geodesic}, each such $q$ is attained along the prefix-signature trace. The resonance argument may select any such factorising point, whereas $p$ is the unique one yielding a reduced literal path conjugation, up to waiting intervals and weak increasing reparametrisation.
\end{remark}

\subsection{Factorisation for one finite frequency set}

\begin{proposition}[Finite-frequency prefix factorisation]\label{prop:finite-frequency-factorisation}
Let $V$ be a finite-dimensional real normed vector space with
$\dim V\geq2$, and let
$\gamma\colon[0,T]\to V$ be a continuous bounded-variation path. Set
\[
l=\log S(\gamma),
\qquad
v=\gamma_T-\gamma_0,
\]
and assume that
\[
R(l)=\infty,
\qquad
v\neq0.
\]
Choose $\ell\in V^*$ such that $\ell(v)=1$, put $W=\ker\ell$, and write
\[
\gamma_t-\gamma_0=x_t v+X_t,
\qquad
x_t=\ell(\gamma_t-\gamma_0),
\qquad
X_t\in W.
\]

Let $\Lambda\subset2\pi\Q$ be finite and symmetric, so that
$\Lambda=-\Lambda$. Give $W$ an auxiliary Euclidean norm, let $W_{\C}$
be its complexification, and define the real Hilbert space
\[
E_\Lambda
=
\set{(z_\omega)_{\omega\in\Lambda}
\in\prod_{\omega\in\Lambda}W_{\C}:
 z_{-\omega}=\overline{z_\omega}}
\]
with the norm inherited from the Hilbert direct sum. Define the decorated
transverse path
\[
B_t^\Lambda(\gamma)
=
\left(
\int_0^t \e^{i\omega x_s}\dd X_s
\right)_{\omega\in\Lambda},
\qquad 0\leq t\leq T,
\]
and the orthogonal map
\[
T_\Lambda\colon E_\Lambda\to E_\Lambda,
\qquad
(T_\Lambda z)_\omega=\e^{i\omega}z_\omega.
\]
Denote by the same symbol the induced graded automorphism of tensor
series, and set
\[
r_t^\Lambda
=S\bigl(B^\Lambda(\gamma)|_{[0,t]}\bigr),
\qquad
h_\Lambda=r_T^\Lambda.
\]

Then there exists $t_\Lambda\in[0,T]$ such that
\begin{equation}\label{eq:fixed-prefix}
r_{t_\Lambda}^\Lambda
=h_\Lambda T_\Lambda\bigl(r_{t_\Lambda}^\Lambda\bigr).
\end{equation}
Equivalently, with
\[
a_\Lambda=r_{t_\Lambda}^\Lambda,
\]
one has
\begin{equation}\label{eq:finite-factorisation}
h_\Lambda
=a_\Lambda T_\Lambda(a_\Lambda)^{-1}.
\end{equation}
\end{proposition}
\begin{proof}
Apply the construction of \cref{sec:resonance} to the centred path $\gamma-\gamma_0$. This path has the same signature as $\gamma$, and its decorated path is exactly the path $B^\Lambda(\gamma)$ defined above. Choose $m\geq1$ such that
\[
m\omega\in2\pi\Z
\qquad (\omega\in\Lambda).
\]
Then $T_\Lambda^m=\id$, and \cref{prop:resonance} gives
\[
h_\Lambda T_\Lambda(h_\Lambda)\cdots
T_\Lambda^{m-1}(h_\Lambda)=1.
\]
On $\Sl(E_{\Lambda})$, define
\[
F_{\Lambda}(p)=h_{\Lambda}T_{\Lambda}(p).
\]
Since $h_{\Lambda}=S(B^{\Lambda}(\gamma))$ belongs to $\Sl(E_{\Lambda})$, the group property in \cref{thm:LDZ} and the preservation statement in \cref{lem:prefix-continuity}(a) show that $F_{\Lambda}$ is a well-defined self-map of $\Sl(E_{\Lambda})$. Left multiplication is an isometry by \cref{thm:LDZ}; $T_{\Lambda}$ is induced by the orthogonal map \cref{eq:T-Lambda}, so it is an isometry by \cref{lem:prefix-continuity}.
Moreover,
\[
F_{\Lambda}^m(p)
 =h_{\Lambda}T_{\Lambda}(h_{\Lambda})\cdots
  T_{\Lambda}^{m-1}(h_{\Lambda})T_{\Lambda}^m(p)
 =p
\]
by \cref{eq:cyclic-identity}.
By \cref{lem:finite-order-isometry}, $F_{\Lambda}$ has a fixed point $p_{\Lambda}$ on
\[
[1,F_{\Lambda}(1)]=[1,h_{\Lambda}].
\]
The prefix trace $\set{r_t^{\Lambda}:0\leq t\leq T}$ is connected by \cref{lem:prefix-continuity} and contains both endpoints. No reducedness of $B^{\Lambda}(\gamma)$ is required: connectedness of its prefix trace is sufficient. A connected subset of a metric tree contains the geodesic between any two of its points \cite[Section~1]{CullerMorgan1987}; hence the prefix trace contains $[1,h_{\Lambda}]$.
Thus $p_{\Lambda}=r_{t_{\Lambda}}^{\Lambda}$ for some $t_{\Lambda}$, and the fixed-point equation gives \cref{eq:fixed-prefix} and \cref{eq:finite-factorisation}.
\end{proof}

\section{One prefix time for all rational frequencies}\label{sec:common-prefix}

For each finite symmetric $\Lambda$, \cref{prop:finite-frequency-factorisation} supplies a nonempty set of admissible prefix times. Compatibility under frequency projection makes these sets nested, and compactness produces one time $t_*$ valid for every rational frequency set.

\begin{proposition}[Compatible prefix time]\label{prop:common-prefix}
Let $V$ be a finite-dimensional real normed vector space with
$\dim V\geq2$, and let
$\gamma\colon[0,T]\to V$ be a continuous bounded-variation path. Set
\[
l=\log S(\gamma),
\qquad
v=\gamma_T-\gamma_0,
\]
and assume that
\[
R(l)=\infty,
\qquad
v\neq0.
\]
Fix $\ell\in V^*$ such that $\ell(v)=1$, put $W=\ker\ell$, and write
\[
\gamma_t-\gamma_0=x_t v+X_t,
\qquad
x_t=\ell(\gamma_t-\gamma_0),
\qquad
X_t\in W.
\]

Give $W$ an auxiliary Euclidean norm and let $W_{\C}$ be its
complexification. For every finite symmetric set
$\Lambda\subset2\pi\Q$, define the real Hilbert space
\[
E_\Lambda
=
\set{(z_\omega)_{\omega\in\Lambda}
\in\prod_{\omega\in\Lambda}W_{\C}:
 z_{-\omega}=\overline{z_\omega}}
\]
with the norm inherited from the Hilbert direct sum, the decorated
transverse path
\[
B_t^\Lambda(\gamma)
=
\left(
\int_0^t \e^{i\omega x_s}\dd X_s
\right)_{\omega\in\Lambda},
\qquad 0\leq t\leq T,
\]
and the orthogonal map
\[
T_\Lambda\colon E_\Lambda\to E_\Lambda,
\qquad
(T_\Lambda z)_\omega=\e^{i\omega}z_\omega.
\]
Denote by the same symbol the induced graded automorphism of tensor
series, and set
\[
r_t^\Lambda
=S\bigl(B^\Lambda(\gamma)|_{[0,t]}\bigr),
\qquad
h_\Lambda=r_T^\Lambda.
\]

Then there exists $t_*\in[0,T]$ such that, for every finite symmetric
$\Lambda\subset2\pi\Q$,
\begin{equation}\label{eq:common-factorisation}
h_\Lambda
=a_\Lambda T_\Lambda(a_\Lambda)^{-1},
\qquad
a_\Lambda=r_{t_*}^\Lambda.
\end{equation}
Equivalently,
\[
r_{t_*}^\Lambda
=h_\Lambda T_\Lambda\bigl(r_{t_*}^\Lambda\bigr)
\]
for every such $\Lambda$.
\end{proposition}
\begin{proof}
For each finite symmetric $\Lambda\subset2\pi\Q$, set
\begin{equation}\label{eq:Z-Lambda}
\begin{aligned}
Z_{\Lambda}
&=\set{t\in[0,T]:
 r_t^{\Lambda}=h_{\Lambda}T_{\Lambda}(r_t^{\Lambda})}\\
&=\set{t\in[0,T]:
 d_{\infty}\!\left(r_t^{\Lambda},
 h_{\Lambda}T_{\Lambda}(r_t^{\Lambda})\right)=0}.
\end{aligned}
\end{equation}
By \cref{prop:finite-frequency-factorisation}, every $Z_{\Lambda}$ is nonempty. The map
\[
t\longmapsto
 d_{\infty}\!\left(r_t^{\Lambda},
 h_{\Lambda}T_{\Lambda}(r_t^{\Lambda})\right)
\]
is continuous by \cref{lem:prefix-continuity}, because $T_{\Lambda}$ and left multiplication by $h_{\Lambda}$ are isometries. Equation~\eqref{eq:Z-Lambda} therefore exhibits $Z_{\Lambda}$ as the inverse image of the closed set $\set{0}$, so $Z_{\Lambda}$ is closed.

If $\Lambda\subset\Lambda'$, let
\[
P_{\Lambda'\to\Lambda}:E_{\Lambda'}\to E_{\Lambda}
\]
be coordinate projection. Directly from the definition of the decorated paths,
\[
P_{\Lambda'\to\Lambda}B_t^{\Lambda'}(\gamma)
 =B_t^{\Lambda}(\gamma)
\qquad (0\leq t\leq T).
\]
By functoriality of the signature, let $P_{\Lambda'\to\Lambda,*}$ denote the induced graded tensor homomorphism. It satisfies
\[
P_{\Lambda'\to\Lambda,*}(h_{\Lambda'})=h_{\Lambda},
\qquad
P_{\Lambda'\to\Lambda,*}(r_t^{\Lambda'})=r_t^{\Lambda}.
\]
Coordinate projection also intertwines the rotations,
\[
P_{\Lambda'\to\Lambda,*}\circ T_{\Lambda'}
 =T_{\Lambda}\circ P_{\Lambda'\to\Lambda,*}.
\]
Therefore $Z_{\Lambda'}\subset Z_{\Lambda}$.

Enumerate the positive elements of $2\pi\Q$ as $\omega_1,\omega_2,\ldots$ and put
\[
\Lambda_n=\set{0,\pm\omega_1,\ldots,\pm\omega_n}.
\]
The sets $Z_{\Lambda_n}$ form a nested sequence of nonempty closed subsets of the compact interval $[0,T]$.
Their intersection is nonempty.
Choose $t_*$ in the intersection.
Every finite symmetric rational frequency set is contained in some $\Lambda_n$, and projection gives \cref{eq:common-factorisation}.
\end{proof}

\section{Reconstruction of the ordinary signature}\label{sec:reconstruction}

The common time $t_*$ from \cref{prop:common-prefix} defines a conjugate-line comparison path $\widetilde\gamma$. We prove $S(\widetilde\gamma)=S(\gamma)$ by viewing each decorated coefficient as the Fourier transform of a finite Stieltjes measure obtained from an ordered simplex. Equality at all rational frequencies recovers these measures, and polynomial functions of the longitudinal coordinates then reconstruct every mixed signature word in the splitting $V=\R v\oplus W$.

Throughout this section, retain the assumptions and notation of \cref{prop:common-prefix}. By translation invariance, replace the original path by its centred version $\gamma-\gamma_0$ and retain the notation $\gamma$. Thus $\gamma_0=0$; the signature, the decomposition $(x,X)$, the decorated paths, and the common prefix time $t_*$ are unchanged.

Let
\[
\alpha=\gamma|_{[0,t_*]}.
\]
Write $\lambda_v:[0,1]\to V$ for the line segment $\lambda_v(s)=sv$ and define the based reverse
\[
\overleftarrow{\alpha}_s
 =\alpha_{t_*-s}-\alpha_{t_*},
\qquad 0\leq s\leq t_*.
\]
Define the conjugate-line path
\begin{equation}\label{eq:comparison-path}
\widetilde\gamma
 =\alpha*\lambda_v*\overleftarrow{\alpha},
\end{equation}
naturally parametrised on an interval $[0,\widetilde T]$, for example $\widetilde T=2t_*+1$.
All signature and decoration constructions are invariant under orientation-preserving reparametrisation, so the particular concatenation parametrisation is immaterial.
Chen's identity gives
\begin{equation}\label{eq:comparison-signature}
S(\widetilde\gamma)=S(\alpha)\e^vS(\alpha)^{-1}.
\end{equation}

Both $\gamma$ and $\widetilde\gamma$ start at zero and have increment $v$.
For each path $\beta\in\{\gamma,\widetilde\gamma\}$, write on its own parameter interval
\[
\beta_t=x_t^{\beta}v+X_t^{\beta},
\qquad
x_t^{\beta}=\ell(\beta_t),
\qquad
X_t^{\beta}\in W,
\]
and define $B^{\Lambda}(\beta)$ using this decomposition.

\subsection{Equality of all decorated signatures}

\begin{proposition}\label{prop:decorated-equality}
For every finite symmetric $\Lambda\subset 2\pi\Q$,
\[
S(B^{\Lambda}(\widetilde\gamma))
 =S(B^{\Lambda}(\gamma)).
\]
\end{proposition}

\begin{proof}
The decorated transverse signature of the initial copy of $\alpha$ is
\[
a_{\Lambda}=r_{t_*}^{\Lambda}.
\]
The middle line has zero transverse differential.
Let
\[
C_s^{\Lambda}
 =B_s^{\Lambda}(\alpha)
 =\left(\int_0^s\e^{i\omega x_u}\dd X_u\right)_{\omega\in\Lambda},
\]
and write $C_{s,\omega}^{\Lambda}$ for its $\omega$-coordinate.
On the appended reverse copy, at reverse time $u\in[0,t_*]$ the point is
\[
v+\alpha_{t_*-u}
 =(1+x_{t_*-u})v+X_{t_*-u}.
\]
For each $\omega\in\Lambda$, the continuous-BV change of variables for reversal gives
\begin{align*}
\int_0^u
 \e^{i\omega(1+x_{t_*-r})}\dd\bigl(X_{t_*-r}\bigr)
&=-\e^{i\omega}\int_{t_*-u}^{t_*}
  \e^{i\omega x_s}\dd X_s\\
&=\e^{i\omega}\bigl(C_{t_*-u,\omega}^{\Lambda}-C_{t_*,\omega}^{\Lambda}\bigr).
\end{align*}
Therefore the full decorated transverse prefix is
\begin{align*}
\left(\int_0^u
 \e^{i\omega(1+x_{t_*-r})}\dd\bigl(X_{t_*-r}\bigr)
\right)_{\omega\in\Lambda}
&=T_{\Lambda}\bigl(C_{t_*-u}^{\Lambda}-C_{t_*}^{\Lambda}\bigr)\\
&=\overleftarrow{\,T_{\Lambda}C^{\Lambda}\,}_u.
\end{align*}
Thus the appended decorated path is exactly the based reverse of $T_{\Lambda}B^{\Lambda}(\alpha)$, and its signature is
\[
S\!\left(\overleftarrow{\,T_{\Lambda}B^{\Lambda}(\alpha)\,}\right)
 =T_{\Lambda}(a_{\Lambda})^{-1}.
\]
By \cref{eq:common-factorisation},
\[
S(B^{\Lambda}(\widetilde\gamma))
 =a_{\Lambda}T_{\Lambda}(a_{\Lambda})^{-1}
 =h_{\Lambda}
 =S(B^{\Lambda}(\gamma)).
\]
\end{proof}

\subsection{Finite Stieltjes measures on ordered simplices}

We use the standard polar decomposition of complex measures and the product-measure and Fubini theorems; see, for example, \cite[Chapters~6 and~8]{RudinRCA1987}.

\begin{lemma}[Atomless Stieltjes product measures]\label{lem:atomless-products}
Let $\nu_1,\ldots,\nu_N$ be finite signed or complex Borel measures on a compact interval $I$, and suppose that each total-variation measure $\abs{\nu_j}$ is atomless.
Then
\[
\abs{\nu_1\otimes\cdots\otimes\nu_N}
=
\abs{\nu_1}\otimes\cdots\otimes\abs{\nu_N}.
\]
Every fixed-coordinate hyperplane $\set{t_i=c}$ and every coordinate-collision hyperplane $\set{t_i=t_j}$ is null for this positive product measure, and hence also for $\nu_1\otimes\cdots\otimes\nu_N$.
Consequently, finite unions of such hyperplanes are null, strict and non-strict ordered-simplex conventions give the same integrals, and coordinate permutations and iterated integrations are justified by Fubini's theorem for finite signed or complex measures.
\end{lemma}

\begin{proof}
By polar decomposition, write $\dd\nu_j=f_j\,\dd\abs{\nu_j}$ with $\abs{f_j}=1$ $\abs{\nu_j}$-almost everywhere.
Then
\[
\dd(\nu_1\otimes\cdots\otimes\nu_N)
 =\prod_{j=1}^N f_j(t_j)\,
  \dd(\abs{\nu_1}\otimes\cdots\otimes\abs{\nu_N}),
\]
which proves the displayed identity of total-variation measures.
For a fixed-coordinate hyperplane, Tonelli's theorem applied to the positive product measure gives a factor $\abs{\nu_i}(\set{c})=0$.
For a collision hyperplane with $i\neq j$, integrate first in the $j$th coordinate: for each fixed value of $t_i$, the corresponding section is the singleton $\set{t_i}$ and has $\abs{\nu_j}$-measure zero.
Tonelli's theorem therefore gives zero again.
The remaining assertions follow because the boundaries separating strict from weak orderings are finite unions of these hyperplanes, while Fubini's theorem applies to every finite signed or complex product measure through its finite total variation.
\end{proof}

Let $\beta:[0,T_{\beta}]\to V$ be a continuous bounded-variation path with $\beta_0=0$ and $\beta_{T_{\beta}}=v$. Using the fixed splitting $V=\R v\oplus W$, write
\[
\beta_t=x_t v+X_t,
\qquad x_t=\ell(\beta_t),
\qquad X_t\in W.
\]
Fix $r\geq 1$ and real covectors $\eta_1,\ldots,\eta_r\in W^*$, and write $\eta=(\eta_1,\ldots,\eta_r)$.
Put
\[
\xi_j(t)=\eta_j(X_t).
\]
Let $\nu_j$ be the finite signed Lebesgue--Stieltjes measure on $[0,T_{\beta}]$ determined by
\[
\nu_j((a,b])=\xi_j(b)-\xi_j(a),
\qquad
\nu_j(\set{0})=0.
\]
Because $\xi_j$ is continuous, $\nu_j$ is atomless. More explicitly, the increasing variation function
\[
q_j(t)=\Var(\xi_j;[0,t])
\]
is continuous, and the Lebesgue--Stieltjes measure induced by $q_j$ is exactly the total-variation measure $\abs{\nu_j}$. Hence $\abs{\nu_j}$ is atomless; in particular, for every $t$,
\begin{equation}\label{eq:variation-atomless}
\abs{\nu_j}(\set{t})=\abs{\nu_j(\set{t})}=0.
\end{equation}

The same observation applies to every scalar coordinate measure used below, including the Stieltjes measure $\nu_x$ induced by $x$.
Thus \cref{lem:atomless-products} applies to every product measure appearing in this section.

Let
\[
\Delta_r(T_{\beta})
 =\set{(t_1,\ldots,t_r)\in[0,T_{\beta}]^r:
 t_1<\cdots<t_r}.
\]
Define the restriction of the finite signed product measure to the strict simplex by
\[
\nu_{\beta}^{\eta}(A)
 = (\nu_1\otimes\cdots\otimes\nu_r)
   \bigl(A\cap\Delta_r(T_{\beta})\bigr)
\]
for Borel $A\subset[0,T_{\beta}]^r$.
Define its pushforward to $\R^r$ by
\begin{equation}\label{eq:pushforward-measure}
\mu_{\beta}^{\eta}
 =(x^{\times r})_*\nu_{\beta}^{\eta},
\qquad
x^{\times r}(t_1,\ldots,t_r)
 =(x_{t_1},\ldots,x_{t_r}).
\end{equation}
Thus, for every bounded Borel function $f$,
\begin{equation}\label{eq:pushforward-integral}
\int_{\R^r} f(y)\dd\mu_{\beta}^{\eta}(y)
 =\int_{\Delta_r(T_{\beta})}
   f(x_{t_1},\ldots,x_{t_r})
   \prod_{j=1}^r\eta_j(\dd X_{t_j}).
\end{equation}
The resulting finite signed Borel measure has total-variation norm
\[
\norm{\mu_{\beta}^{\eta}}_{\mathrm{TV}}
 :=\abs{\mu_{\beta}^{\eta}}(\R^r)
 \leq \prod_{j=1}^r\Var(\xi_j;[0,T_{\beta}]),
\]
and its support lies in the compact set $x([0,T_{\beta}])^r$.

\subsection{Fourier recovery}

For a finite complex Borel measure $\mu$ on $\R^r$, use the Fourier-transform convention
\[
\widehat\mu(\omega)=\int_{\R^r}\e^{i\omega\cdot y}\dd\mu(y),
\qquad \omega\cdot y=\sum_{j=1}^r\omega_jy_j.
\]

\begin{lemma}[Fourier recovery]\label{lem:fourier-recovery}
Let $\beta^{(1)}:[0,T_1]\to V$ and $\beta^{(2)}:[0,T_2]\to V$ be continuous bounded-variation paths, translated to start at zero, with the same nonzero increment $v$.
Fix one functional $\ell\in V^*$ with $\ell(v)=1$, put $W=\ker\ell$, and write
\[
\beta_t^{(k)}=x_t^{(k)}v+X_t^{(k)},
\qquad x_t^{(k)}=\ell(\beta_t^{(k)}),
\qquad X_t^{(k)}\in W
\quad(k=1,2).
\]
Equip $W$ with an auxiliary Euclidean norm, define the decorated paths by \cref{eq:decorated-path}, and define the measures by \cref{eq:pushforward-measure}.
If
\[
S(B^{\Lambda}(\beta^{(1)}))
 =S(B^{\Lambda}(\beta^{(2)}))
\]
for every finite symmetric $\Lambda\subset 2\pi\Q$, then
\[
\mu_{\beta^{(1)}}^{\eta}
 =\mu_{\beta^{(2)}}^{\eta}
\]
for every $r\geq1$ and every tuple $\eta=(\eta_1,\ldots,\eta_r)\in(W^*)^r$.
\end{lemma}

\begin{proof}
For each $k\in\{1,2\}$, \cref{eq:pushforward-integral} gives
\begin{equation}\label{eq:fourier-coordinate}
\widehat\mu_{\beta^{(k)}}^{\eta}(\omega_1,\ldots,\omega_r)
 =\int_{\Delta_r(T_k)}
   \prod_{j=1}^r
   \e^{i\omega_jx_{t_j}^{(k)}}
   \eta_j(\dd X_{t_j}^{(k)}).
\end{equation}
If all $\omega_j\in2\pi\Q$, choose a finite symmetric $\Lambda$ containing them. The signatures lie in the real tensor algebra over $E_{\Lambda}$, and equality there remains valid after complexification. They may therefore be paired with the complex-linear coordinate functionals $\zeta_{\omega,\eta}$ from \cref{lem:decorated-coordinate-separation} and their tensor products. Under this pairing, \cref{eq:fourier-coordinate} is exactly the corresponding signature coordinate. The hypothesis therefore gives equality of the Fourier transforms on the dense set $(2\pi\Q)^r$.
Fourier transforms of finite measures are continuous, so they agree on all of $\R^r$.
Fourier uniqueness for finite complex measures gives equality of the measures.
\end{proof}

\subsection{Recovery of every mixed word}

\medskip
\noindent\textbf{No monotonicity assumption.}
The scalar coordinate $x$ need not be monotone. The reduction of each consecutive $\dd x$-block to a power of its net increment follows only from continuity, atomlessness of the Stieltjes total-variation measure, and permutation symmetry of the product measure.

\begin{lemma}[One-dimensional block integral]\label{lem:block-integral}
Let $x:[0,T]\to\R$ be continuous and of bounded variation, and let $0\leq a<b\leq T$.
For $n\geq 0$,
\begin{equation}\label{eq:block-integral}
\int_{a<s_1<\cdots<s_n<b}
\dd x_{s_1}\cdots\dd x_{s_n}
 =\frac{(x_b-x_a)^n}{n!}.
\end{equation}
\end{lemma}

\begin{proof}
Let $\nu_x$ be the signed Stieltjes measure induced by $x$.
By \cref{lem:atomless-products}, all diagonals in $(a,b)^n$ are $\abs{\nu_x}^{\otimes n}$-null.
For $n\geq 1$, the $n!$ strict coordinate orderings of $(a,b)^n$ therefore partition the product domain up to a null set.
The product measure $\nu_x^{\otimes n}$ is invariant under coordinate permutations, so every strict ordering has the same mass, namely $1/n!$ of the total mass of $(a,b)^n$.
Because there are no endpoint atoms,
\[
\nu_x((a,b))=x_b-x_a.
\]
This proves \cref{eq:block-integral}.
The case $n=0$ is interpreted as $1$.
\end{proof}

\begin{lemma}[Mixed-word formula]\label{lem:mixed-word}
Let $\beta:[0,T_{\beta}]\to V$ be a continuous bounded-variation path with $\beta_0=0$ and $\beta_{T_{\beta}}=v\neq0$. Fix $\ell\in V^*$ with $\ell(v)=1$, put $W=\ker\ell$, and write
\[
\beta_t=x_tv+X_t,
\qquad x_t=\ell(\beta_t),
\qquad X_t\in W.
\]
Let $r\geq1$, let $n_0,\ldots,n_r\geq 0$, and let $\eta_1,\ldots,\eta_r\in W^*$, extended to $V^*$ by $\eta_j(v)=0$. Set $N=r+\sum_{j=0}^r n_j$. The canonical pairing of the covector word with the degree-$N$ signature level is
\begin{equation}\label{eq:mixed-word}
\begin{aligned}
&\left\langle
\ell^{\otimes n_0}\otimes\eta_1\otimes
\ell^{\otimes n_1}\otimes\cdots\otimes
\eta_r\otimes\ell^{\otimes n_r},
S_N(\beta)
\right\rangle\\
&\quad=\int_{\Delta_r(T_{\beta})}
\frac{x_{t_1}^{n_0}}{n_0!}
\prod_{j=1}^{r-1}
\frac{(x_{t_{j+1}}-x_{t_j})^{n_j}}{n_j!}
\frac{(1-x_{t_r})^{n_r}}{n_r!}
\prod_{j=1}^r\eta_j(\dd X_{t_j}).
\end{aligned}
\end{equation}
\end{lemma}

\begin{proof}
Let $\nu_x$ be the finite signed Stieltjes measure induced by $x$, and let $\nu_j$ be the signed Stieltjes measure induced by $t\mapsto\eta_j(X_t)$.
For each $j$, introduce longitudinal variables $s_{j,1},\ldots,s_{j,n_j}$, omitting the block when $n_j=0$.
In the coordinate order prescribed by the covector word, define the finite signed product measure
\[
M=\nu_x^{\otimes n_0}\otimes\nu_1\otimes
  \nu_x^{\otimes n_1}\otimes\cdots\otimes
  \nu_r\otimes\nu_x^{\otimes n_r}
\]
on $[0,T_\beta]^N$.
The desired signature coefficient is $M(D)$, where $D$ is the strict-order domain
\begin{align*}
0&<s_{0,1}<\cdots<s_{0,n_0}<t_1
 <s_{1,1}<\cdots<s_{1,n_1}<t_2<\cdots\\
&\hspace{7em}<t_r<s_{r,1}<\cdots<s_{r,n_r}<T_\beta,
\end{align*}
with empty strings omitted.

All factors have finite total variation and atomless total variation.
By \cref{lem:atomless-products}, the boundary hyperplanes are $\abs{M}$-null, strict versus non-strict boundary conventions do not affect $M(D)$, and Fubini may be applied without boundary terms.

Permute coordinates so that the transverse variables $t_1,\ldots,t_r$ are outer variables.
For $t=(t_1,\ldots,t_r)\in\Delta_r(T_\beta)$, set $t_0=0$ and $t_{r+1}=T_\beta$.
The section of $D$ in the longitudinal variables is the Cartesian product of the ordered simplexes
\[
\set{t_j<s_{j,1}<\cdots<s_{j,n_j}<t_{j+1}},
\qquad 0\leq j\leq r.
\]
Fubini's theorem therefore gives the coefficient as
\begin{equation}\label{eq:mixed-word-fubini}
\int_{\Delta_r(T_\beta)}
 \prod_{j=0}^r
 \left(
  \int_{t_j<s_1<\cdots<s_{n_j}<t_{j+1}}
  \dd x_{s_1}\cdots\dd x_{s_{n_j}}
 \right)
 \prod_{k=1}^r\eta_k(\dd X_{t_k}),
\end{equation}
where an inner integral with $n_j=0$ equals $1$.
By \cref{lem:block-integral}, the $j$th inner factor is
\[
\frac{(x_{t_{j+1}}-x_{t_j})^{n_j}}{n_j!}.
\]
Since $x_{t_0}=x_0=0$ and $x_{t_{r+1}}=x_{T_\beta}=1$, substituting these factors into \cref{eq:mixed-word-fubini} gives \cref{eq:mixed-word}.
The argument requires only continuity and bounded variation of $x$.
\end{proof}

\paragraph{Worked reconstruction example.}
Take $r=2$, $n_0=1$, $n_1=2$, and $n_2=0$ in \cref{lem:mixed-word}. Then $N=5$ and
\begin{equation}\label{eq:worked-reconstruction}
\begin{aligned}
&\left\langle
\ell\otimes\eta_1\otimes\ell^{\otimes2}\otimes\eta_2,
S_5(\beta)
\right\rangle\\
&\qquad=
\int_{0<t_1<t_2<T_\beta}
 x_{t_1}\,
 \frac{(x_{t_2}-x_{t_1})^2}{2}
 \,\eta_1(\dd X_{t_1})\eta_2(\dd X_{t_2}).
\end{aligned}
\end{equation}
The longitudinal letter before $\eta_1$ contributes
\[
\int_{0<s<t_1}\dd x_s=x_{t_1},
\]
while the two longitudinal letters between $\eta_1$ and $\eta_2$ contribute, by \cref{lem:block-integral},
\[
\int_{t_1<s_1<s_2<t_2}\dd x_{s_1}\dd x_{s_2}
 =\frac{(x_{t_2}-x_{t_1})^2}{2}.
\]
The final block is empty and contributes $1$. By the pushforward identity \cref{eq:pushforward-integral}, \cref{eq:worked-reconstruction} is equivalently
\[
\int_{\R^2}
 y_1\frac{(y_2-y_1)^2}{2}
 \,\dd\mu_\beta^{(\eta_1,\eta_2)}(y_1,y_2).
\]
Thus equality of the measures $\mu_\beta^{(\eta_1,\eta_2)}$ recovers this signature coefficient directly. No monotonicity of $x$ is used.

\begin{theorem}[Decorated signatures determine the ordinary signature]\label{thm:decorated-determines}
Let $\beta^{(1)}:[0,T_1]\to V$ and $\beta^{(2)}:[0,T_2]\to V$ be continuous bounded-variation paths, translated to start at zero, with the same nonzero increment $v$.
Fix $\ell\in V^*$ with $\ell(v)=1$, put $W=\ker\ell$, equip $W$ with an auxiliary Euclidean norm, and define the decompositions and decorated paths as in \cref{lem:fourier-recovery}.
If
\[
S(B^{\Lambda}(\beta^{(1)}))
 =S(B^{\Lambda}(\beta^{(2)}))
\]
for every finite symmetric $\Lambda\subset 2\pi\Q$, then
\[
S(\beta^{(1)})=S(\beta^{(2)}).
\]
\end{theorem}

\begin{proof}
Let $d=\dim V$ and choose a basis $\eta^1,\ldots,\eta^{d-1}$ of $W^*$, extended by $\eta^k(v)=0$.
Then
\[
\set{\ell,\eta^1,\ldots,\eta^{d-1}}
\]
is a basis of $V^*$, and its tensor words form a basis of every $(V^*)^{\otimes n}$.

Every tensor word in the displayed basis is of exactly one of the following two types. It is either purely longitudinal, or, after listing its $r\geq1$ transverse letters in their original order as $\eta_1,\ldots,\eta_r$, it has the unique block decomposition
\[
\ell^{\otimes n_0}\otimes\eta_1\otimes
\ell^{\otimes n_1}\otimes\cdots\otimes
\eta_r\otimes\ell^{\otimes n_r},
\qquad n_0,\ldots,n_r\geq0.
\]
Thus \cref{lem:mixed-word} covers every non-purely-longitudinal basis word, including adjacent transverse letters through blocks with $n_j=0$.

By \cref{lem:fourier-recovery}, the measures $\mu_{\beta^{(1)}}^{\eta}$ and $\mu_{\beta^{(2)}}^{\eta}$ agree for every tuple of transverse covectors. For fixed $n_0,\ldots,n_r$, put $N=r+\sum_{j=0}^r n_j$ and define
\[
P_{n_0,\ldots,n_r}(y_1,\ldots,y_r)
 =\frac{y_1^{n_0}}{n_0!}
  \prod_{j=1}^{r-1}
  \frac{(y_{j+1}-y_j)^{n_j}}{n_j!}
  \frac{(1-y_r)^{n_r}}{n_r!},
\]
where the empty product is interpreted as $1$. For either $\beta\in\set{\beta^{(1)},\beta^{(2)}}$, \cref{eq:pushforward-integral,eq:mixed-word} give
\[
\begin{aligned}
&\left\langle
 \ell^{\otimes n_0}\otimes\eta_1\otimes
 \ell^{\otimes n_1}\otimes\cdots\otimes
 \eta_r\otimes\ell^{\otimes n_r},
 S_N(\beta)
\right\rangle\\
&\hspace{35mm}
 =\int_{\R^r}P_{n_0,\ldots,n_r}(y)\,
   \dd\mu_{\beta}^{\eta}(y).
\end{aligned}
\]
Equality of the pushforward measures therefore gives equality of every mixed-word coefficient for the two paths.
The purely longitudinal case corresponds to $r=0$; both paths have longitudinal increment $\ell(v)=1$, and \cref{lem:block-integral} gives $1/n!$ at degree $n$.
The tensor words in the displayed basis separate every tensor level, so the full signatures agree.
\end{proof}

\section{Proof of the classification theorem}

We now assemble the preceding results to prove \cref{thm:main}.

\subsection{The nonzero-increment implication}

Retain the notation $\overline\gamma$, $g$, $l$, and $v$ from \cref{thm:main}, and assume $v\neq0$. By translation invariance, replace $\gamma$ by the centred path $\overline\gamma$ and again denote it by $\gamma$; then $\gamma_0=0$, while $g$, $l$, and $v$ are unchanged. If $\dim V=1$, \cref{rem:one-dimensional} gives $S(\gamma)=\e^v$, and the conclusion holds with $t_*=0$. Hence assume $\dim V\geq2$.

By \cref{prop:common-prefix}, there is a time $t_*$ and a prefix $\alpha=\gamma|_{[0,t_*]}$. Define $\widetilde\gamma$ by \cref{eq:comparison-path}. By \cref{prop:decorated-equality,thm:decorated-determines},
\[
S(\gamma)=S(\widetilde\gamma).
\]
Using \cref{eq:comparison-signature},
\[
S(\gamma)=S(\alpha)\e^vS(\alpha)^{-1}.
\]
Translating back, $\alpha$ is the prefix of the centred path required in \cref{thm:main}. The zero-increment conclusion is \cref{thm:loop-rigidity}.

\subsection{The converse implication}

\begin{proposition}[Conjugacy to a line gives an entire logarithm]\label{prop:converse}
Let $V$ be a finite-dimensional real normed space, let $\alpha$ be a based continuous bounded-variation $V$-valued path, put $a=S(\alpha)$, and let $v\in V$. Then
\[
g=a\e^v a^{-1}
\]
has an entire logarithmic signature.
\end{proposition}

\begin{proof}
By Chen's identity,
\[
g=S(\alpha*\lambda_v*\overleftarrow{\alpha}),
\]
so $g$ is itself the signature of a bounded-variation path. In the completed graded algebra,
\[
a\e^v a^{-1}=\e^{a v a^{-1}},
\]
so
\[
\log g=a v a^{-1}=\Ad_a v.
\]
Let $A=L(\alpha)$. If $A=0$, then $\alpha$ is constant, $a=1$, and $\log g=v$, so the conclusion is immediate. Assume henceforth that $A>0$.
The inverse $a^{-1}$ is the signature of the reversed path and satisfies the same factorial estimate as $a$.
Writing $a_p$ and $(a^{-1})_q$ for homogeneous components, the degree-$n$ component is
\[
(\log g)_n
 =\sum_{\substack{p,q\geq0\\p+q=n-1}} a_p\otimes v\otimes(a^{-1})_q
\qquad(n\geq1).
\]
Therefore
\begin{align*}
\norm{(\log g)_n}_{\pi,n}
&\leq \norm{v}
  \sum_{\substack{p,q\geq0\\p+q=n-1}}\frac{A^p}{p!}\frac{A^q}{q!}\\
&=\norm{v}\,\frac{(2A)^{n-1}}{(n-1)!}.
\end{align*}
The $n$th roots tend to zero, so the homogeneous radius is infinite.
\end{proof}

The preceding sections prove the signature-level assertions in \cref{thm:main}. If $\gamma$ is tree-reduced, \cref{thm:reduction-lemma} converts the factorisation into the path identity \cref{eq:path-level-main} modulo weak increasing reparametrisation; in the zero-increment case, Hambly--Lyons uniqueness makes the tree-reduced path constant. The converse follows from \cref{prop:converse}. This proves \cref{cor:modified-conjecture}.

\section{Conclusion}\label{sec:conclusion}

The classification identifies an analytically defined locus in the rectifiable signature group with a simple conjugacy geometry:
\[
R(\log g)=\infty
\quad\Longleftrightarrow\quad
 g=a\e^v a^{-1}.
\]
Here $v=g_1$ is the degree-one component of $g$, and the zero-first-level fibre consists only of the identity. For tree-reduced representatives, minimum-displacement axis geometry turns the algebraic normal form into the corresponding path decomposition.

The proof uses two mechanisms. The spectral-growth theorem converts exponential-type control of a matrix exponential into polynomial degree control of the characteristic invariants of its entire logarithm. The remaining argument closes resonant finite-dimensional identities through tree fixed points, compactness, and Fourier reconstruction. Bounded variation enters at two points: it supplies finite atomless Stieltjes measures and places the signature in the rectifiable complete metric tree. An extension to general weakly geometric rough paths would require substitutes for both mechanisms.

\appendix

\section{Metric-tree facts}\label{app:metric-tree}

We record the standard complete-real-tree facts used in \cref{sec:tree-factorisation}; see also \cite[Section~1]{CullerMorgan1987}. For points $x,y$ in a metric tree, $[x,y]$ denotes their unique geodesic segment, and for a nonempty subset $C$ we write $d(x,C)=\inf_{c\in C}d(x,c)$.

\begin{lemma}[Projection and gate decomposition]\label{lem:projection-gates}
Let $(\mathscr{T},d)$ be a complete metric tree and let $C\subset\mathscr{T}$ be nonempty, closed, and convex. Every $x\in\mathscr{T}$ has a unique metric projection $p\in C$. Moreover, $p\in[x,c]$ for every $c\in C$. If $p$ and $q$ are the projections of $x$ and $y$, respectively, and $p\neq q$, then
\begin{equation}\label{eq:gate-decomposition}
[x,y]=[x,p]\cup[p,q]\cup[q,y],
\end{equation}
where consecutive pieces meet only at their common endpoint and $[x,p]\cap[q,y]=\varnothing$.
\end{lemma}

\begin{proof}
Choose $c_n\in C$ with $d(x,c_n)\downarrow d(x,C)$. If $b_{nm}$ is the branch point of $x,c_n,c_m$, then $b_{nm}\in C$ and
\[
d(c_n,c_m)=d(x,c_n)+d(x,c_m)-2d(x,b_{nm}),
\]
so $(c_n)$ is Cauchy. Completeness and closedness give a minimiser; two distinct minimisers would have a branch point in $C$ strictly closer to $x$, proving uniqueness.

If $p$ is the projection of $x$ and $c\in C$, the branch point of $x,p,c$ lies in $C$, so minimality forces it to be $p$; hence $p\in[x,c]$. For projections $p\neq q$ of $x$ and $y$, this gives $p\in[x,q]$ and $q\in[y,p]$. The two off-$C$ segments cannot meet away from their endpoints, since a common point would have both $p$ and $q$ as projections onto $C$. Their union with $[p,q]\subset C$ is therefore an arc from $x$ to $y$, and hence the unique geodesic.
\end{proof}

\begin{lemma}[Finite-order isometry]\label{lem:finite-order-isometry}
Let $F$ be an isometry of a complete metric tree $(\mathscr{T},d)$ and suppose $F^m=\id$ for some $m\geq1$. For every $o\in\mathscr{T}$, $F$ has a fixed point on $[o,F(o)]$.
\end{lemma}

\begin{proof}
The convex hull of the finite orbit $\{o,F(o),\ldots,F^{m-1}(o)\}$ is a compact finite tree. On it, the maximum distance to the orbit has a unique minimiser. Indeed, if distinct points $x,x'$ both had minimal radius $R_0$ and $c$ were the midpoint of $[x,x']$, then for every orbit point $y$ the midpoint identity in a metric tree is
\[
d(c,y)=\max\{d(x,y),d(x',y)\}-\frac12d(x,x').
\]
Consequently,
\[
d(c,y)\leq R_0-\frac12d(x,x')<R_0,
\]
which contradicts minimality after taking the maximum over the orbit. Since $F$ preserves the orbit, it fixes this minimiser. Thus the fixed-point set
\[
\Fix(F):=\set{x\in\mathscr{T}:F(x)=x}
\]
is nonempty; it is also closed and convex. Let $p$ be the projection of $o$ onto $\Fix(F)$ from \cref{lem:projection-gates}. If $[p,o]$ and $[p,F(o)]$ met beyond $p$, uniqueness of points at a fixed distance from $p$ would make that common point fixed, contradicting minimality of $p$. Thus their union is $[o,F(o)]$, and $p$ is the required fixed point.
\end{proof}

\begin{proposition}[Axis formula for a metric-tree translation]\label{prop:metric-tree-axis}
Let $(\mathscr{T},d)$ be a complete metric tree and let $L:\mathscr{T}\to\mathscr{T}$ be a surjective isometry. Suppose that a closed geodesic line $\mathcal{A}\subset\mathscr{T}$ has an isometric parametrisation $\alpha_{\mathcal A}:\R\to\mathcal{A}$ satisfying
\[
L(\alpha_{\mathcal A}(t))=\alpha_{\mathcal A}(t+\tau)\qquad(t\in\R)
\]
for some $\tau>0$. For $x\in\mathscr{T}$, let $q$ be the projection of $x$ onto $\mathcal{A}$. Then $Lq$ is the projection of $Lx$ and
\begin{equation}\label{eq:abstract-axis-decomposition}
[x,Lx]=[x,q]\cup[q,Lq]\cup[Lq,Lx],
\end{equation}
with intersections only at consecutive endpoints. Moreover,
\begin{equation}\label{eq:abstract-axis-formula}
d(x,Lx)=2d(x,\mathcal{A})+\tau,
\end{equation}
and
\begin{equation}\label{eq:abstract-min-displacement-set}
\mathcal{A}=\{x\in\mathscr{T}:d(x,Lx)=\tau\}.
\end{equation}
Thus $\mathcal{A}$ is the unique minimum-displacement axis of $L$.
\end{proposition}

\begin{proof}
Because $L\mathcal{A}=\mathcal{A}$, the point $Lq$ is the projection of $Lx$. Also $d(q,Lq)=\tau$. Applying \cref{lem:projection-gates} to $x$ and $Lx$ gives \cref{eq:abstract-axis-decomposition}; the first and last pieces have the common length $d(x,\mathcal{A})$. Summing their lengths proves \cref{eq:abstract-axis-formula}, and equality with $\tau$ holds exactly on the closed line $\mathcal{A}$, proving \cref{eq:abstract-min-displacement-set}.
\end{proof}




\section*{Declaration of generative AI and AI-assisted technologies in the manuscript preparation process}
During the development and preparation of this work, the author used OpenAI's ChatGPT 5.6 Sol as a conversational aid for exploring possible proof routes, critically testing intermediate arguments, organising the exposition, and assisting with drafting and language editing. The author independently reconstructed and checked the mathematical arguments, reviewed and revised all AI-assisted material, verified the cited sources, and takes full responsibility for the content of the publication.

\end{document}